\documentclass[a4paper,12pt]{article}

\makeatletter
\@addtoreset{footnote}{page}
\makeatother

\usepackage{enumitem}

\usepackage{amsthm}
\usepackage{amsmath,amssymb,latexsym,amsfonts,mathrsfs}

\newcommand{\eps}{\ensuremath{\varepsilon}}

\renewcommand{\hat}{\widehat}
\renewcommand{\tilde}{\widetilde}
\renewcommand{\bar}{\overline}

\newcommand{\bD}{\ensuremath{\mathbb{D}}}
\newcommand{\bE}{\ensuremath{\mathbb{E}}}

\newcommand{\bN}{\ensuremath{\mathbb{N}}}

\newcommand{\bR}{\ensuremath{\mathbb{R}}}

\newcommand{\cB}{\ensuremath{\mathcal{B}}}

\newcommand{\cL}{\ensuremath{\mathcal{L}}}
\newcommand{\cM}{\ensuremath{\mathcal{M}}}

\theoremstyle{plain}
\newtheorem{Thm}{Theorem}[section]

\newtheorem{Prop}[Thm]{Proposition}

\theoremstyle{definition}

\newtheorem{Def}[Thm]{Definition}
\newtheorem{Rem}[Thm]{Remark}

\numberwithin{equation}{section}

\makeatletter
\renewcommand\section{\@startsection {section}{1}{\z@}%
                                   {-3.5ex \@plus -1ex \@minus -.2ex}%
                                   {2.3ex \@plus.2ex}%
                                   {\normalfont\large\bf}}
\makeatother

\makeatletter
\renewcommand\subsection{\@startsection {subsection}{1}{\z@}%
                                   {-3.5ex \@plus -1ex \@minus -.2ex}%
                                   {2.3ex \@plus.2ex}%
                                   {\normalfont\normalsize\bf}}
\makeatother

\begin{document}

\begin{center}
{\Large \bf 
Fluctuation scaling limits for positive recurrent jumping-in diffusions with large jumps 
}
\end{center}
\begin{center}
Kosuke Yamato (Kyoto University)
\end{center}

\begin{center}
	{\small \today}
\end{center}

\begin{abstract}
	For positive recurrent jumping-in diffusions with large jumps, we study scaling limits of the fluctuations of inverse local times and occupation times.
	We generalize the eigenfunctions with modified Neumann boundary condition, which have been introduced in the previous study for small jumps, to a class of diffusions of stronger singularity.
\end{abstract}


\section{Introduction}\label{section: intro}

Let us consider a strong Markov process $X$ on the half line $[0,\infty)$ (resp.\ the real line $\bR$) which has continuous paths up to the first hitting time of $0$ and, as soon as $X$ hits $0$, $X$ jumps into the interior $(0,\infty)$ (resp.\ $\bR \setminus \{0\}$) and starts afresh. We call such a process $X$ a {\it unilateral} (resp.\ {\it bilateral}) {\it jumping-in diffusion}.

As we will see in Section \ref{section: contthmofinverselocaltime}, unilateral jumping-in diffusions are specified by two Radon measures on $(0,\infty)$, the speed measure $m$ which characterize the diffusive behavior on $(0,\infty)$ and the jumping-in measure $j$ which do the jumps from the origin.
Let $X$ be a positive recurrent unilateral jumping-in diffusion with large jumps, that is, $\int_{0}^{\infty}xj(dx) = \infty$.
Positive recurrence ensures existence of the following limit of the inverse local time $\eta$ at $0$:
\begin{align}
	\frac{1}{t}\eta (t) &\xrightarrow[t \to \infty]{P} b  \quad (b \in [0,\infty)).  \label{eq95}
\end{align}
Our main aim is to study the {\it fluctuation scaling limit} of $\eta$:
\begin{align}
	f(\gamma)\left(\frac{\eta (\gamma t)}{\gamma} - b t\right) &\xrightarrow[\gamma \to \infty]{d} B(\kappa t) \quad \mathrm{on} \  \bD  \label{eq92}
\end{align}
for some scaling function $f$, a constant $\kappa > 0$ and a standard Brownian motion $B$, where $\bD$ denotes the space of c\`adl\`ag paths from $[0,\infty)$ to $\bR$ equipped with Skorokhod's $J_1$-topology.
When $\eta$ is of finite variance, this scaling limit can be easily obtained by the classical central limit theorem with $f(\gamma) = \gamma^{1/2}$. Hence our main concern is the case of infinite variance.  
	

In the previous paper \cite{YamatoYano}, the author has studied positive recurrent jumping-in diffusions with small jumps: $\int_{0}^{\infty}xj(dx) < \infty$.
One of the main results in \cite{YamatoYano} was to show the fluctuation scaling limit \eqref{eq92} when the limit process is an $\alpha$-stable process $S_\alpha$ ($1 < \alpha < 2$) without negative jumps, that is, 
\begin{align}
f(\gamma)\left(\frac{\eta (\gamma t)}{\gamma} - b t\right) &\xrightarrow[\gamma \to \infty]{d} S_\alpha(t) \quad \mathrm{on} \  \bD.  \label{eq96}
\end{align}
The main tool was the Krein-Kotani correspondence, which gives a one-to-one and bi-continuous correspondence between a class of speed measures and a class of Herglotz functions. To apply the correspondence, it was necessary to assume
\begin{align}
\int_{0+}m(x,1)^2dx < \infty \label{eq48}
\end{align}
for the speed measure $m$ of jumping-in diffusions. Under \eqref{eq48}, we may define the eigenfunction for $\frac{d}{dm}\frac{d}{dx}$ with {\it modified Neumann boundary condition} at $0$ and the Laplace exponent of $\eta$ may be represented by the eigenfunctions and the Herglotz function corresponding to $m$. The fluctuation scaling limit \eqref{eq96} was proven by showing a continuity of them w.r.t.\ $m$ and $j$.
In the present paper, we allow exit boundaries with stronger singularity, including those that do not satisfy \eqref{eq48}, and construct eigenfunctions subject to more singular boundary conditions.

As an application of \eqref{eq92}, we consider the fluctuation scaling limit of the occupation time of a bilateral jumping-in diffusion $X$: $A(t) = \int_{0}^{t}1_{(0,\infty)}(X_s)ds$. In the positive recurrent case where $A(t)$ has a degenerate mean:
\begin{align}
\frac{1}{t}A(t) \xrightarrow[t \to \infty]{P} p \in (0,1), \label{eq155} 
\end{align}
we establish the fluctuation scaling limits of the form:
\begin{align}
f(\gamma)
\left(\frac{A(\gamma t)}{\gamma} - p t\right) \xrightarrow[\gamma \to \infty]{f.d.} B( \kappa t) \label{eq154}
\end{align}
for some constant $\kappa$ and a function $f(\gamma)$ in \eqref{eq92}. Here $\xrightarrow{f.d.}$ denotes the convergence of finite-dimensional distributions.

\subsection{Main results} \label{section:mainresults}

According to some results of Feller \cite{Feller:Theparabolic} and It\^o \cite{Ito:PPP}, under the natural scale, a unilateral jumping-in diffusion can be characterized by its speed measure $m$ and its jumping-in measure $j$, both of which are Radon measures on $(0,\infty)$. The jumping-in diffusion is a strong Markov process on $[0,\infty)$ which behaves as a $\frac{d}{dm}\frac{d^+}{dx}$-diffusion during staying in $(0,\infty)$ and jumps from the origin to $(0,\infty)$ according to $j$, where $\frac{d^+}{dx}$ denotes the right-differentiation operator. In Section \ref{section: contthmofinverselocaltime}, we will give the precise description. We denote the jumping-in diffusion by $X_{m,j}$ and its inverse local time at $0$ by $\eta_{m,j}$. 

We introduce conditions on $m$ and $j$ for our main results.
Let $K$ and $L$ be functions from $(0,\infty)$ to $\bR$ which vary slowly at $\infty$. Consider the following conditions for $\alpha > 1$ and $\beta > 0$:
\begin{align}
\mathrm{(M)}_{\alpha,K}&: \ m(x,\infty) \sim (\alpha - 1)^{-1}x^{1/\alpha - 1}K(x) \quad (x \to \infty), \label{} \\ 
\mathrm{(J)}_{\beta,L}&: \ j(x,\infty) \sim x^{-\beta}L(x) \quad (x \to \infty), \label{}
\end{align}
where $f(x) \sim g(x) \ (x \to \infty)$ means $\lim_{x \to \infty}f(x) / g(x) = 1$.
Note that the case $\beta \leq 1$ corresponds to the large jump case $\int_{0}^{\infty}xj(dx) = \infty$ and the case $\beta > 1$ does to the small jump case $\int_{0}^{\infty}xj(dx) < \infty$.  

In this paper, we show the fluctuation scaling limit \eqref{eq92} under $\mathrm{(M)}_{\alpha,K}$ and $\mathrm{(J)}_{\beta,L}$ with $\alpha \geq 2$ and $\alpha \beta = 2$.
These conditions obviously implies $\beta \leq 1$ and thus the process $X_{m,j}$ has large jumps: $\int_{0}^{\infty}xj(dx) = \infty$.
The reason why we only consider such $\alpha$ and $\beta$ is as follows: the case $1 < \alpha < 2$ and $\beta > 1$ (small jump case) has already been treated in \cite{YamatoYano} (we will review in more detail in Section \ref{section:PrevStudies}) and the case $\alpha > 2$ and $\alpha \beta > 2$ is easy because the variance of $\eta_{m,j}$ is finite, and we can appeal to the classical central limit theorem. 

Now we explain our main results.
Let $d(m)$ denote the constant of singularity of $m$ near $0$, which we will introduce in \eqref{eq124} and let $n_{m,j}$ denote an excursion measure of $X_{m,j}$ away from $0$, whose precise form is in \eqref{excursion_measure}.
A necessary and sufficient condition for the existence of $X_{m,j}$ will be given in Section \ref{section: contthmofinverselocaltime}.
The following Theorems \ref{informal-main-alpha2+} and \ref{informal-main-alpha-int} are one of our main results, which gives the fluctuation scaling limit of the inverse local time $\eta_{m,j}$ under $\mathrm{(M)}_{\alpha,K}$ and $\mathrm{(J)}_{\beta,L}$ with $\alpha \geq 2$ and $\alpha \beta = 2$ in the case of $\alpha$ is integer or not, respectively.
The proof of Theorem \ref{informal-main-alpha2+} will be given in Section \ref{section: scalinglimit} and that of Theorem \ref{informal-main-alpha-int} will be given in Appendix \ref{appendix}.
Examples for these theorems will be given in Section \ref{section:examle}.

\begin{Thm} \label{informal-main-alpha2+}
	Let $\alpha > 2$ be not an integer and assume $X_{m,j}$ exists. Suppose the following holds
	\begin{enumerate}
		\item $d(m) < \infty$.
		\item $\mathrm{(M)}_{\alpha,K}$ holds.
		\item $\mathrm{(J)}_{2/\alpha,L}$ holds.
	\end{enumerate}
	Define
	\begin{align}
	N(\gamma) := \int_{0}^{\gamma}j(dx)\int_{0}^{x}dy\int_{y}^{\gamma}dm(z)\int_{0}^{z}m(w,\infty)dw. \label{eq197}
	\end{align}
	Let $u$ and $v$ be slowly varying functions at $\infty$ such that 
	\begin{align}
		u(\gamma)^2v(\gamma) \sim N(\gamma) \quad (\gamma \to \infty) \quad \text{and} \quad	\lim_{\gamma \to \infty}\frac{K(\gamma)}{u(\gamma)} = \lim_{\gamma \to \infty}\frac{L(\gamma)}{v(\gamma)} = 0. \label{eq174}
	\end{align}
	Then we have
	\begin{align}
		\frac{1}{\gamma^{1/2}u(\gamma^{\alpha/2})}\left(\eta_{m,j}\left(\frac{\gamma t}{v(\gamma^{\alpha/2})}\right) - b\frac{\gamma t}{v(\gamma^{\alpha/2})}\right) \xrightarrow[\gamma \to \infty]{d}B(2 t) \quad \mathrm{on} \ \bD, \label{eq218}
	\end{align}
	and
	\begin{align}
		n_{m,j}[T_0 > s ] = o(s^{-2}U^{\sharp}(s^2)^{-1}v(s^{\alpha}U^{\sharp}(s^2)^{\alpha/2})) \quad (s \to \infty), \label{eq217}
	\end{align}
	where $b := \int_{0}^{\infty}j(dx)\int_{0}^{x}m(y,\infty)dy$, $U^{\sharp}$ is a de Bruijn conjugate of $U(s) = u(s^{\alpha/2})$ (see e.g., \cite[p.29]{Regularvariation} for the definition of de Bruijn conjugate) and $T_0$ is the lifetime of an excursion path (see Section \ref{section: contthmofinverselocaltime}).
\end{Thm}

\begin{Rem}\label{ex. of u v}
	When $\alpha$ is not an integer, we always have the functions $u$ and $v$ satisfying \eqref{eq174}.
	As we will show in Proposition \ref{comparison of N,K,L alpha > 2}, it holds
	\begin{align}
		\lim_{\gamma \to \infty}\frac{K(\gamma)^2L(\gamma)}{N(\gamma)} = 0. \label{}
	\end{align}
	Thus, for example, the following $u$ and $v$ satisfy the desired condition:
	\begin{align}
		u(\gamma) = S(\gamma)^{p/2}K(\gamma), \quad v(\gamma) = S(\gamma)^{1-p}L(\gamma) \label{}
	\end{align}
	for $S(\gamma) := N(\gamma) / K(\gamma)^2L(\gamma)$ and $p \in (0,1)$.
\end{Rem}

\begin{Rem}
	The limit result \eqref{eq218} is equivalent to the following form which is consistent with \eqref{eq92}:
	\begin{align}
		f(\gamma)\left(\frac{\eta_{m,j}(\gamma t)}{\gamma} - b\right) \xrightarrow[\gamma \to \infty]{d}B(2\kappa t) \quad  \mathrm{on} \ \bD \label{}
	\end{align}
	with
	\begin{align}
		f(\gamma) = \frac{1}{u(\gamma^{\alpha/2}w^{\sharp}(\gamma)^{\alpha/2})}\sqrt{\frac{\gamma}{w^{\sharp}(\gamma)}}, \label{}
	\end{align}
	where $w^{\sharp}$ is a de Bruijn conjugate of $w(\gamma) = v(\gamma)^{-1}$.
\end{Rem}

In the case where $\alpha$ is integer, we need some additional assumptions.
\begin{Thm} \label{informal-main-alpha-int}
	Let $\alpha \geq 2$ be an integer and assume $X_{m,j}$ exists. 
	Suppose the conditions (i) - (iii) in Theorem \ref{informal-main-alpha2+} holds.
	Let $u$ and $v$ be slowly varying functions at $\infty$ satisfying \eqref{eq174}
	for $N$ in \eqref{eq197}.
	Assume the following:
	\begin{enumerate}
		\setcounter{enumi}{3}
		\item $\lim_{\gamma \to \infty}\frac{K(\gamma)^{d - \alpha  + 1}
		\int_{1}^{\gamma}\frac{K(x)^{\alpha}}{x}dx}{u(\gamma)^{d+1}}
		= 0 \quad $ for some $d > 0$.
	\end{enumerate}
	In the case of $\alpha = 2$, we also assume
	\begin{enumerate}
		\setcounter{enumi}{4}
		\item $\limsup_{\gamma \to \infty} \frac{1}{v(\gamma)}\int_{1}^{\gamma}\frac{L(x)}{x}dx < \infty $.
	\end{enumerate}
	Then we have
	\begin{align}
		\frac{1}{\gamma^{1/2}u(\gamma^{\alpha/2})}\left(\eta_{m,j}\left(\frac{\gamma t}{v(\gamma^{\alpha/2})}\right) - b\frac{\gamma t}{v(\gamma^{\alpha/2})}\right) \xrightarrow[\gamma \to \infty]{d}B(2 t) \quad \mathrm{on} \ \bD, \label{eq218d}
	\end{align}
	and
	\begin{align}
		n_{m,j}[T_0 > s ] = o(s^{-2}U^{\sharp}(s^2)^{-1}v(s^{\alpha}U^{\sharp}(s^2)^{\alpha/2})) \quad (s \to \infty), \label{eq217d}
	\end{align}
	where $b := \int_{0}^{\infty}j(dx)\int_{0}^{x}m(y,\infty)dy$ and $U^{\sharp}$ is a de Bruijn conjugate of $U(s) = u(s^{\alpha/2})$.
\end{Thm}

\subsection{Scaling limits of occupation times of bilateral jumping-in diffusions}

As an application of Theorem \ref{informal-main-alpha2+}, we establish the fluctuation scaling limit of the form \eqref{eq154} for the occupation times on the half line of bilateral jumping-in diffusions.

As we will see in Section \ref{section: contthmofinverselocaltime}, a bilateral jumping-in diffusion is characterized by two pairs of the speed measures and jumping-in measures $(m_\pm,j_\pm)$.
For a bilateral jumping-in diffusion $X = X_{m_+,j_+;m_-,j_-}$, we denote its occupation time on the positive side by
\begin{align}
A(t) := \int_{0}^{t}1_{(0,\infty)}(X_s)ds. \label{}
\end{align}
We focus on the case where the limit ratio degenerates, that is,
\begin{align}
\lim_{t \to \infty}\frac{1}{t}A(t) \xrightarrow[t \to \infty]{P} p \in (0,1). \label{}
\end{align}

The following two theorems give the fluctuation scaling limit of $A(t)$.
We always write $\tilde{B}$ for an independent copy of $B$.

\begin{Thm} \label{informal-occ-main-alpha2+}
	Let $\alpha > 2$ be not an integer and assume $X_{m_+,j_+;m_-,j_-}$ exists.
	Suppose the following holds for constants $w_\pm > 0$ and $r_\pm > 0$:
	\begin{enumerate}
		\item $d(m_\pm) < \infty$.
		\item $m_\pm$ satisfies $\mathrm{(M)}_{\alpha,w_\pm K}$.
		\item $j_\pm$ satisfies $\mathrm{(J)}_{2/\alpha,r_\pm^2 L}$.
	\end{enumerate}
	We define $N(\gamma)$ as \eqref{eq197} for 
	\begin{align}
	m(\gamma,\infty) := w_+^{-1}m_+(\gamma,\infty)\quad \text{and} \quad j(\gamma,\infty) := r_+^{-2}j_+(\gamma,\infty). \label{5}
	\end{align}
	Let $u$ and $v$ be slowly varying functions at $\infty$ such that
	\begin{align}
	u(\gamma)^2v(\gamma) \sim N(\gamma) \quad (\gamma \to \infty) \quad \text{and}
	\quad  
	\lim_{\gamma \to \infty}\frac{K(\gamma)}{u(\gamma)} = \lim_{\gamma \to \infty}\frac{L(\gamma)}{v(\gamma)} = 0. \label{eq224}
	\end{align}
	Then we have
	\begin{align}
	\frac{1}{\gamma^{1/2}u(\gamma^{\alpha/2})}\left(A\left(\frac{\gamma t}{v(\gamma^{\alpha/2})}\right) - \frac{p\gamma t}{v(\gamma^{\alpha/2})}\right) \xrightarrow[\gamma \to \infty ]{f.d.} (1-p)w_+r_+B(2 t) - pw_-r_-\tilde{B}(2 t), \label{}
	\end{align}
	where 
	\begin{align}
	b_\pm &= 
	\int_{0}^{\infty}j_\pm(dx)\int_{0}^{x}m_\pm(y,\infty)dy \quad \text{and} \quad p = \frac{b_+}{b_+ + b_-}.  \label{}
	\end{align}
\end{Thm}

Similar to Theorem \ref{informal-main-alpha-int},
we need additional assumptions when $\alpha$ is an integer.

\begin{Thm} \label{informal-occ-main-alpha-int}
	Let $\alpha \geq 2$ be an integer and assume $X_{m_+,j_+;m_-,j_-}$ exists.
	Suppose the conditions (i)-(iii) in Theorem \ref{informal-occ-main-alpha2+} for constants $w_\pm > 0$ and $r_\pm > 0$ holds. Let $u$ and $v$ be slowly varying functions at $\infty$ satisfying \eqref{eq224} for $N(\gamma)$ in Theorem \ref{informal-occ-main-alpha2+}. Assume the following:
	\begin{enumerate}
		\setcounter{enumi}{3}
		\item $\lim_{\gamma \to \infty}\frac{K(\gamma)^{d - \alpha  + 1}
		\int_{1}^{\gamma}\frac{K(x)^{\alpha}}{x}dx}{u(\gamma)^{d+1}}
		= 0 \quad $ for some $d > 0$.
	\end{enumerate}
	When $\alpha = 2$, we also assume
	\begin{enumerate}
		\setcounter{enumi}{4}
		\item $\limsup_{\gamma \to \infty} \frac{1}{v(\gamma)}\int_{1}^{\gamma}\frac{L(x)}{x}dx < \infty$.
	\end{enumerate}
	Then we have
	\begin{align}
	\frac{1}{\gamma^{1/2}u(\gamma^{\alpha/2})}\left(A\left(\frac{\gamma t}{v(\gamma^{\alpha/2})}\right) - \frac{p\gamma t}{v(\gamma^{\alpha/2})}\right) \xrightarrow[\gamma \to \infty ]{f.d.} (1-p)w_+r_+B(2 t) - pw_-r_-\tilde{B}(2 t), \label{}
	\end{align}
	where 
	\begin{align}
	b_\pm &= 
	\int_{0}^{\infty}j_\pm(dx)\int_{0}^{x}m_\pm(y,\infty)dy \quad \text{and} \quad p = \frac{b_+}{b_+ + b_-}.  \label{}
	\end{align}
\end{Thm}

The proof of Theorem \ref{informal-occ-main-alpha2+} could be given by reducing the fluctuation of the occupation time of a bilateral jumping-in diffusion to that of the inverse local times of unilateral jumping-in diffusions via It\^o's excursion theory.
The key to the reduction is the tail behavior of the L\'evy measure of the inverse local time \eqref{eq217}, which is given by applying a Tauberian theorem to the convergence \eqref{eq218}.
These arguments are essentially the same as the one in the proof of Theorem 7.1 of \cite{YamatoYano}.
Therefore we omit the proofs of Theorem \ref{informal-occ-main-alpha2+}. 

 \subsection{Eigenfunctions with modified Neumann boundary condition}\label{subsec: strategy}

%


Our method is similar to that of \cite{YamatoYano}: we reduce the fluctuation scaling limit of $\eta_{m,j}$ to continuity w.r.t.\ $m$ and $j$.
There is, however, a significant difference.
As we have already mentioned in Section \ref{section: intro}, in \cite{YamatoYano} the assumption \eqref{eq48} was indispensable.
In the present paper, in order to treat speed measures which not necessarily satisfy \eqref{eq48}, we introduce a class of speed measures and a notion of convergence for the class in Definition \ref{convofstring}.

We explain our method more precisely.
Let $\chi_{m,j}$ denote the Laplace exponent of $\eta_{m,j}$, that is, $E[\mathrm{e}^{-\lambda \eta_{m,j}(1)}] = \mathrm{e}^{-\chi_{m,j}(\lambda)}$. It can be represented, as we will see in Section \ref{section: contthmofinverselocaltime}, as
\begin{align}
\chi_{m,j}(\lambda) = \int_{0}^{\infty}(1 - g_m(\lambda;x))j(dx), \label{}
\end{align}
where $u = g_m$ is the unique solution to the ODE
\begin{align}
\frac{d}{dm}\frac{d^+}{dx}u = \lambda u, \quad u(0) = 1, \quad u \text{ is non-increasing}. \label{}
\end{align}
To show the convergence \eqref{eq92}, it is enough to prove that of Laplace exponents of L\'evy processes:
\begin{align}
\tilde{\eta}_{m,j,\gamma}(t) := f(\gamma)\left(\frac{\eta_{m,j} (\gamma t)}{\gamma} - b t\right) \quad (\gamma > 0)
\end{align}
as $\gamma \to \infty$. By changing variables, the Laplace exponent $\tilde{\chi}_{m,j,\gamma}$ of $\tilde{\eta}_{m,j,\gamma}$ may be represented by
\begin{align}
\tilde{\chi}_{m,j,\gamma} (\lambda) &= \gamma \int_{0}^{\infty}\left(1 - g_m\left(\frac{f(\gamma)}{\gamma}\lambda; x\right)\right)j(dx) - bf(\gamma)\lambda \label{} \\
&= \int_{0}^{\infty}(1 - g_{m_\gamma}(\lambda; x)) j_\gamma(dx) - b_\gamma\lambda \label{eq194} \\
&= \chi_{m_{\gamma},j_{\gamma}}(\lambda) - b_\gamma \lambda \label{}
\end{align}
for appropriate Radon measures $m_\gamma$ and $j_\gamma$ and a constant $b_\gamma$.
Therefore our problem is reduced to the continuity of the Laplace exponent $\chi_{m,j}$ with respect to $m$ and $j$. 

To analyze the function $g_m$, we use eigenfunctions of initial value problems.
When the boundary $0$ for $dm$ is regular, we have a unique solution $u = \psi_m(\lambda;\cdot)$ to
\begin{align}
\frac{d}{dm}\frac{d^+}{dx}u = \lambda u, \quad  u(0) = 0, \quad  u^+(0) = 1 \label{}
\end{align}
and a unique solution $u = \varphi_m(\lambda;\cdot)$ to
\begin{align}
\frac{d}{dm}\frac{d^+}{dx}u = \lambda u, \quad  u(0) = 1, \quad  u^+(0) = 0, \label{}
\end{align}
where $u^+$ denotes a right-derivative of $u$. 
When the boundary $0$ for $dm$ is exit, we still have $\psi_{m}$ but do not $\varphi_m$. We would like to introduce a counterpart for $\varphi_m$.

In the unilateral case, the speed measure $m$ comes from a string $m$, i.e., $m: (0,\infty) \to \bR$ is a non-decreasing, right-continuous function: $m(a,b] = m(b) - m(a)$. 
In \cite{YamatoYano}, for a string $m$ with
\begin{align}
\int_{0+}m(x)^2dx < \infty, \label{eq176}
\end{align}
we introduced the eigenfunction $u = \varphi^1_m(\lambda;\cdot)$ of the differential equation $\frac{d}{dm}\frac{d^+}{dx}u = \lambda u \ (\lambda > 0)$ with the {\it modified Neumann boundary condition} at $0$:
\begin{align}
u(0)= 1, \quad \lim_{x \to +0}(u^+(x) - \lambda (m(x) - m(1))) = 0. \label{eq177}
\end{align}
When \eqref{eq176} does not hold, the solution for \eqref{eq177} does not exist. 
Hence we introduce more general boundary condition at $0$ and construct the solution for it.



For a string $m$, we define
\begin{align}
G^1_m(x) = \int_{0}^{x}(m(y) - m(1))dy
\end{align}
and inductively define for $k \geq 2$
\begin{align}
G^k_m(x) = -\int_{0}^{x}dy\int_{y}^{1}G^{k-1}_m(z)dm(z). \label{eq196}
\end{align}
Note that $(-1)^kG^k$ is non-negative on $[0,1]$.
We define
\begin{align}
d(m) = \inf\left\{k \geq 1 \;\middle|\; \int_{0}^{1}(-1)^kG^k_m(x)dm(x) < \infty \right\}, \label{eq124} 
\end{align}
where we interpret $\inf \emptyset = \infty$.
We also note that
\begin{align}
\int_{0}^{1}(-1)^kG^k_m(x)dm(x)
\begin{cases}
= \infty & k < d(m), \\
< \infty & k \geq d(m).
\end{cases}
\end{align}
Roughly speaking, the value $d(m)$ is larger when a string $m$ diverges faster at $0$.  
For a string $m$ with $d(m) < \infty$ and an integer $d \geq d(m)$, we can prove the existence of an eigenfunction $u = \varphi^d_m(\lambda ;x) \ (\lambda > 0)$ with the {\it modified Neumann boundary condition} at $0$ defined as follows:
\begin{Def}
	Let $m$ be a string with $d(m) < \infty$, $d \geq d(m) \vee 1$ and $\lambda > 0$.
	We say that $u$ is the $\lambda$-eigenfunction with the modified Neumann boundary condition at $0$ of order $d$ when it holds
	\begin{align}
	\frac{d}{dm} \frac{d^+}{dx}u = \lambda u , \quad u(0) = 1 \label{}
	\end{align}
	and
	\begin{align}
	\lim_{x \to +0} \left(u^+(x) - \left(\lambda m(x) + \sum_{k=1}^{d-1}\lambda^k \int_{x}^{1}G^k_m(y)dm(y)\right)\right) = 0. \label{}
	\end{align}
	We denote the function $u$ by $\varphi^d_m$.
\end{Def}
We will construct $\varphi^d_m$ in Section \ref{section: constofphi}.
Then for a suitable constant $c^d_m(\lambda)$, we have the following expression
\begin{align}
g_m(\lambda;x) = \varphi^d_m(\lambda;x) - c^d_m(\lambda) \psi_m(\lambda;x), \label{eq159}
\end{align}
and we may exploit $g_m$ through $\varphi^d_m$ and $\psi_m$.

 \subsection{Previous studies} \label{section:PrevStudies}
 Here we recall several previous studied related to our main results.

 \subsubsection*{Previous study on scaling limits for diffusions} \label{Prev. diff.}

In order to study the inverse local times and occupation times for diffusions (without any jumps at all), we may assume without loss of generality that diffusions have a natural scale; whose local generators are of the form $\frac{d}{dm}\frac{d^+}{dx}$ with a {\it speed measure} $m$.  

Kasahara and Watanabe \cite{KasaharaWatanabe:Brownianrepresentation} has shown the scaling limit of inverse local times for unilateral diffusions. 

\begin{Thm}[{Kasahara and Watanabe \cite[Theorem 3.7]{KasaharaWatanabe:Brownianrepresentation}}] \label{Thm: Kasahara and Watanebe SL for ILT}
	Let $m$ be a finite Radon measure on $[0,\infty)$ with full support and let $\eta$ be the inverse local time at $0$ of $\frac{d}{dm}\frac{d^+}{dx}$-diffusion on $[0,\infty)$.
	Define
	\begin{align}
		K(\gamma) := \int_{0}^{\gamma}m(x,\infty)^2dx \label{} 
	\end{align}
	and assume $K$ varies slowly at $\infty$ as $\gamma \to \infty$.
	Then we have
	\begin{align}
	f(\gamma)\left(\frac{\eta(\gamma t)}{\gamma} - b t\right) \xrightarrow[\gamma \to \infty]{d} B(2 t) \quad \text{on}\ \bD, \label{eq162}
	\end{align}
	where 
	\begin{align}
		f(\gamma) = \sqrt{\frac{\gamma}{K(\gamma)}} \quad \text{and} \quad \ b = m[0,\infty). \label{}
	\end{align}
\end{Thm}

Applying this result, they showed the scaling limit of the occupation time for bilateral diffusions (without any jumps at all).
\begin{Thm}[{Kasahara and Watanabe \cite[Theorem 4.4]{KasaharaWatanabe:Brownianrepresentation}}] \label{Thm: Kasahara and Watanabe SL for OT}
	Let $m$ be a finite Radon measure on $\bR$ with full support and let $X$ be a $\frac{d}{dm}\frac{d^+}{dx}$-diffusion.
	Define the occupation time on the positive side:
	\begin{align}
		A(t) := \int_{0}^{t}1_{(0,\infty)}(X_s)ds. \label{}
	\end{align}
	Assume the following holds: 
	\begin{align}
	\int_{0}^{\gamma}m(x,\infty)^2dx \sim \sigma_+^2 K(\gamma), \quad  \int_{0}^{\gamma}m(-\infty,-x)^2dx \sim \sigma_-^2 K(\gamma) \quad (\gamma \to \infty). \label{}
	\end{align}
	Then we have
	\begin{align}
	g(\gamma)\left(\frac{A(\gamma t)}{\gamma} - p t\right) \xrightarrow[\gamma \to \infty]{f.d.} \sqrt{2}(1-p)\sigma_+ B(t) - \sqrt{2}p\sigma_- \tilde{B}(t), \label{}
	\end{align}
	where $p = \frac{m(0,\infty)}{m(-\infty,0)}$ and $g(\gamma) = \sqrt{\frac{\gamma m(\bR)}{K(\gamma)}}$.
\end{Thm}

Their idea of the proof was following. For a string $m$ with
\begin{align}
	\int_{0+}m(x)^2dx < \infty, \label{eq161}
\end{align}
i.e., $\int_{0}^{\delta}m(x)^2dx < \infty$ for some $\delta> 0$,
they constructed the inverse local time $T(m;t)$ at $0$ of a unilateral $\frac{d}{dm}\frac{d^+}{dx}$-diffusion in a generalized sense, and showed a kind of continuity of $T(m;t)$ w.r.t. $m$.
Then they reduced the scaling limit to the continuity.

%

\subsubsection*{Previous study on scaling limits for jumping-in diffusions} \label{Prev.JD.}
Yano \cite{Yano:Convergenceofexcursion} has studied scaling limits of jumping-in diffusions, whose one of the main results is the following: 
\begin{Thm}[{Yano \cite[Theorem 2.6]{Yano:Convergenceofexcursion}}]\label{Thm:Yano1}
	Let $\alpha > 1$ and $0 < \beta < 1/\alpha$.
	Assume $\mathrm{(M)}_{\alpha,K}$ and $\mathrm{(J)}_{\beta,L}$ holds.
	Then (with some technical conditions) we have
	\begin{align}
	\frac{1}{\gamma}X_{m,j}(\gamma^{1/\alpha} K(\gamma) t) &\xrightarrow[\gamma \to \infty]{d} X_{m^{(\alpha)}, j^{(\beta)}}(t) \quad \mathrm{on} \ \bD, \label{sl3} \\
	\frac{1}{\gamma^{1/\alpha} K(\gamma)}\eta_{m,j}\left(\frac{\gamma^{\beta} t}{L(\gamma)}\right) &\xrightarrow[\gamma \to \infty]{d} S^{(\alpha\beta)}(t) \quad \mathrm{on} \ \bD, \label{sl2}
	\end{align}
	where $m^{(\alpha)}(x,\infty) \propto x^{1/\alpha - 1}$ and $j^{(\beta)}(dx) \propto \beta x^{- \beta - 1}dx$ so that $X_{m^{(\alpha)}, j^{(\beta)}}$ is an $\alpha \beta$-self-similar jumping-in diffusion and $S^{(\alpha\beta)}$ is an $\alpha\beta$-stable subordinator. By the symbol $\propto$, we mean that the both sides coincide up to a multiplicative constant.
\end{Thm}

This result says under $\mathrm{(M)}_{\alpha,K}$ and the $\mathrm{(J)}_{\beta,L}$ for $\alpha > 1$ and $0 < \alpha \beta < 1$, the scaling limit of the process $X_{m,j}$ converges to a jumping-in diffusion. When $\alpha > 1$ and $\alpha \beta > 1$, there does not exist scaling limits of $X_{m,j}$. Under this condition, however, it is still possible to give scaling limits for the inverse local time $\eta_{m,j}$. In \cite{YamatoYano}, they have shown the scaling limit of $\eta_{m,j}$ exists when $1 < \alpha < 2$:
 \begin{Thm}[{Yamato and Yano \cite[Theorem 6.4]{YamatoYano}}]\label{main-informal-LC1}
	Let $\alpha \in (1,2)$ and assume $X_{m,j}$ exists. Suppose $\mathrm{(M)}_{\alpha,K}$ and the following holds:
	\begin{align}
	\int_{0+}m(x,\infty)^2dx < \infty \quad \text{and}  \quad \int_{0}^{\infty}xj(dx) < \infty \label{}
	\end{align}
	hold.
	Then we have:
	\begin{align}
	&\frac{1}{\gamma^{1/\alpha}K(\gamma)}(\eta_{m,j}(\gamma t) - b\gamma t) \xrightarrow[\gamma \to \infty]{d}S^{(\alpha)}(t) \quad \mathrm{on} \ \bD, \label{}
	\end{align}
	where $S^{(\alpha)}$ is a spectrally positive strictly $\alpha$-stable process and $b = \int_{0}^{\infty}j(dx)\int_{0}^{x}m(y,\infty)dy$.
\end{Thm}

 \subsection{Outline of the paper}
 The remainder of the present paper is organized as follows:
In Section \ref{section: constofphi}, we construct the eigenfunctions $\varphi^d_m$ for $\frac{d}{dm}\frac{d}{dx}$ subject to the modified Neumann boundary condition of order $d$ at $0$ and establish some basic estimates for them.
In Section \ref{section: contthmofinverselocaltime}, we show a continuity of inverse local times with respect to their speed measures and jumping-in measures.
In Section \ref{section: scalinglimit}, we study the fluctuation scaling limit of inverse local times of jumping-in diffusion and give a proof of Theorem \ref{informal-main-alpha2+}.
In Section \ref{section:examle}, we give examples of our main results.

\subsubsection*{Acknowledgements} 
 The author would like to thank Kouji Yano and Shin'ichi Kotani, who read an early draft of the present paper and gave him valuable comments. Thanks to them, the present paper was significantly improved.

 This work was supported by JSPS KAKENHI Grant Number JP21J11000 and JSPS Open Partnership Joint Research Projects Grant Number JPJSBP120209921 and the Research Institute for Mathematical Sciences, an International Joint Usage/Research Center located in Kyoto University and, was carried out under the ISM Cooperative Research Program (2020-ISMCRP-5013).
 

 \section{Construction of eigenfunctions of the generator}\label{section: constofphi}
  
  We prepare some notation.
  A function $m: (0,\infty) \to (-\infty,\infty)$ is called a {\it string} when $m$ is non-decreasing and right-continuous.
  We denote as $\cM$ the set of strings $m$ which are strictly increasing and satisfy
  	\begin{align}
  		\int_{0+}xdm(x) < \infty. \label{eq178}
  \end{align} 
 Note that the condition \eqref{eq178} means the boundary $0$ for $\frac{d}{dm}\frac{d^+}{dx}$-diffusion is regular or exit in the sense of Feller.

 For $m \in  \cM$, we define
 \begin{align}
 	G_m(x) &= \int_{0}^{x}m(y)dy \quad (x \geq 0) , \label{} \\
 	\tilde{m}(x) &= m(x) - m(1)  \quad (x > 0) , \label{} \\
 	G^1_m(x) &= \int_{0}^{x}\tilde{m}(y)dy \quad (x \geq 0). \label{}
 \end{align}
 For a bounded variation function $U$ on $(0,\infty)$ and for a function $f$ with $\int_{0}^{x}|f||dU| < \infty \ (x > 0)$, we denote
 \begin{align}
 	U\bullet f (x) = \int_{0}^{x}f(y)dU(y), \label{}
 	\end{align}
where $|dU|$ is the total variation measure of the Stieltjes measure $dU$. 

 \begin{Prop}\label{estofpsi}

 	For $m \in \cM$, define 
 	\begin{align}
 	\psi_m(\lambda; x) &:= \sum_{k=0}^{\infty}\lambda^k ((s\bullet m \bullet)^k s)(x) \quad (\lambda \in \bR,\   x \geq 0), \label{eq115} \\
 	g_m(\lambda;x) &:= \psi_m(\lambda;x)\int_{x}^{\infty}\frac{dy}{\psi_m(\lambda;y)^2} \quad  (\lambda \in \bR,\  x \geq 0), \label{eq8}
 	\end{align}
 	where $(s\bullet m \bullet)^1 s = s \bullet m \bullet s$, $(s\bullet m \bullet)^2 s = s \bullet m \bullet s \bullet m \bullet s$, etc. Then the following holds for every $x \geq 0$, $d \geq 0$ and $\lambda \geq 0$:
 	\begin{align}
 	(s\bullet m \bullet)^d s(x) &\leq xE^d_m(0;x), \label{eq120} \\
 	m \bullet (s \bullet m \bullet )^{d-1} s(x) &\leq E^d_m(0;x), \label{eq121} \\
 	0 \leq \psi_m(\lambda;x) - \sum_{k=0}^{d-1}\lambda^k(s\bullet m \bullet)^k s(x) &\leq x |\lambda|^d E^d_m(|\lambda|;x), \label{eq122} \\
 	0 \leq \psi^+_{\lambda}(m;x) - 1 - \sum_{k=0}^{d-1}\lambda^{k+1} m \bullet (s \bullet m \bullet)^k s(x) &\leq |\lambda|^{d+1} E^{d}_m(|\lambda|;x), \label{eq123}
 	\end{align} 
 	where $(m \bullet s)^d(x) = \left( \int_{0}^{x}y dm(y) \right)^d$ and $E^d_m(\lambda;x) = (1/d!)(m \bullet s)^d (x) \mathrm{e}^{\lambda (m \bullet s)(x)}$.
 \end{Prop}
 The proof of this proposition is given in \cite[Proposition 3.4]{YamatoYano}. Hence we omit it.
 \begin{Rem}\label{charofpsi}
 	\begin{enumerate}
 		\item The function $u = \psi_m(\lambda;\cdot)$ is the unique solution of the integral equation:
 		\begin{align}
 		u(x) = x + \lambda \int_{0}^{x}(x-y)u(y)dm(y) \quad \text{for} \  \lambda \in \bR \quad \text{and} \quad  x \in [0,\infty). \label{}
 		\end{align}
 		In other words, the function $u = \psi_{m}(\lambda; \cdot)$ is the unique solution of the ODE $\frac{d}{dm}\frac{d^+}{dx}u = \lambda u$ satisfying the boundary condition $u(0) = 0$ (Dirichlet) and $u^+(0) = 1$.
 		\item The function $u = g_m(\lambda;\cdot)$ is the unique, non-negative and non-increasing solution of the equation $\frac{d}{dm}\frac{d^+}{dx}u = \lambda u$ satisfying the boundary condition $u(0+) = 1$ and $\lim_{x \to \infty}\frac{d^+}{dx}u(x) = 0$.
 		In fact, since we have
 		\begin{align}
 		g_m \psi^+_{m} - g^+_m\psi_m
 		&= g_m\psi^+_{m} - \psi^+_{m}\left(\int_{x}^{\infty}\frac{dy}{\psi_{m}(y)^2}\right)\psi_{m} + 1 = 1, \label{}
 		\end{align}
 		it follows that
 		\begin{align}
 		0 = d(g_m \psi^+_{m} - g^+_m\psi_m) 
 		= g_m d\psi^+_{m} - \psi_{m}dg^+_m 
 		= \lambda g_m \psi_{m}dm - \psi_{m}dg^+_m. \label{} 			
 		\end{align}
 		Thus we obtain
 		\begin{align}
 		\frac{d}{dm}\frac{d^+}{dx}g_m = \lambda g_m \label{}
 		\end{align}
 		(this argument is due to It\^o \cite{Ito:Essentialsof}).
 	\end{enumerate}
 \end{Rem}
 
  \begin{Prop}\label{integrability}
 	Let $m \in \cM$. For $k \geq 2$, the function $G^k_m(x)$ given in \eqref{eq196} is finite for every $x \geq 0$.
 \end{Prop}
 \begin{proof}
 	For every $k \geq 1$, it is not difficult to check that if $|G^k_m(x)| < \infty$ for some $x > 0$, then $|G^k_m(x)| < \infty$ for every $x > 0$. Therefore it is enough to show the following: for every $k \geq 1$, 
 	\begin{align}
 		(-1)^kG^k_m(x) \ \text{is finite and non-decreasing on} \  [0,1]. \label{eq65}
 	\end{align}
	The assertion \eqref{eq65} obviously holds for $k=1$ from the definition of $\cM$. Assume \eqref{eq65} holds for an integer $k \geq 1$. 
 	By integrating by parts, for $x \in [0,1]$, we have
 	\begin{align}
 		(-1)^{k+1}G^{k+1}_m(x) &= (-1)^k\int_{0}^{x}dy\int_{y}^{1}G^k_m(z)dm(z) \label{}\\
 		&= (-1)^{k+1}\int_{0}^{x}dy\left( G^k_m(y)\tilde{m}(y)  + \int_{y}^{1}(G^k_m)^+(z)\tilde{m}(z)dz  \right) \label{}\\
 		&\leq (-1)^{k+1}G^k_m(1)G^1_m(x) + (-1)^{k+1} \int_{0}^{x}\tilde{m}(y)dy\int_{y}^{1}(G^k_m)^+(z)dz \label{} \\
 		&\leq 2(-1)^{k+1}G^k_m(1)G^1_m(x) < \infty. \label{}
 	\end{align}
 	Therefore by induction, we obtain the desired result.
 \end{proof}
 We introduce a subset of $\cM$ as follows.
 \begin{Def}
  	For $m \in \cM_0 := \{ m \in \cM \mid \lim_{x \to +0}m(x) > -\infty \}$ we define $d(m) = 0$. For $m \in \cM \setminus \cM_0$, we define
 	\begin{align}
 		d(m) = \inf\left\{k \geq 1 \;\middle|\; \int_{0}^{1}(-1)^kG^k_m(x)dm(x) < \infty \right\},
 	\end{align}
 	where $\inf \emptyset = \infty$.
 \end{Def}
 Here we introduce the $\lambda$-eigenfunction announced in the beginning of this section.
 \begin{Def}
	For $m \in \cM$ with $d(m) < \infty$, $d \geq d(m)\vee 1$ and $x \geq 0$, we define
	\begin{align}
		\varphi^d_m(\lambda;x) = 1 + \sum_{k=1}^{d}\lambda^{k} G^{k}_m(x) + \sum_{k=1}^{\infty}\lambda^{d+k}(s\bullet m \bullet)^k G^{d}_m(x). \label{eq9}
	\end{align} 
 \end{Def}
 The convergence of the summation in RHS of \eqref{eq9} follows from the following two propositions. Since the proof is almost the same as Proposition \ref{estofpsi}, we omit it.
  
 \begin{Prop}\label{estofphi^d}
 	Let $m \in \cM$ with $d(m) < \infty$. Then for any $d \geq d(m)$ and $k \geq 1$, the following holds:
 	\begin{align}
 	|(s\bullet m \bullet)^k G^{d}_m(x)| &\leq 
 	xS^d_m(x)E^{k-1}_m(0;x), \label{}\\
 	|(m\bullet s \bullet)^{k-1} m \bullet G^{d}_m(x)| &\leq 
 	S^d_m(x)E^{k-1}_m(0;x), \label{} \\
 	\left|\varphi^d_m(\lambda;x) - 1 - \sum_{k=1}^{d}\lambda^{k} G^{k}_m(x)\right| &\leq 
 	x\lambda^{d+1}S^d_m(x) E^0_m(\lambda;x), \label{eq183} \\
 	\left|(\varphi^d_m)^+(\lambda;x) - \lambda \tilde{m}(x) - \sum_{k=2}^{d}\lambda^{k}\int_{x}^{1} G^{k-1}_m(y)dm(y)\right| &\leq 
 	S^d_m(x) \lambda^{d+1} E^0_m(\lambda;x), \label{}
 	\end{align} 
 	where $S^d_m(x) = \sup_{y \in [0,x]}|m \bullet G^d_m(y)|$.
 \end{Prop}
 We have the following characterization of $\varphi^d_m$. The proof is immediate from Proposition \ref{estofphi^d}. Hence we omit it.
 \begin{Prop} \label{rem1}
 	For $m \in \cM$ with $d(m) < \infty$ and for $d \geq d(m)\vee 1$, the following holds: 
	\begin{enumerate}
		\item The summation in \eqref{eq9} converges uniformly on every compact subset of $[0,\infty)$.
		\item \label{integralequation} The function $u = \varphi^d_m(\lambda;\cdot) - (1+\sum_{k=1}^{d-1}\lambda^k G^k_m(\cdot))$ is the unique solution of the integral equation
		\begin{align}
			u(x) = \lambda^{d}G^{d}_m(x) + \lambda\int_{0}^{x}(x-y)u(y)dm(y). \label{eq36}
		\end{align}
		Equivalently, the function $u=\varphi^d_m(x;\cdot)$ is the unique solution of the equation $\frac{d}{dm}\frac{d^+}{dx}u = \lambda u$ for $\lambda > 0$ with the boundary condition:
		\begin{align}
			u(0) = 1, \quad \lim_{x \to +0} \left(u^+(x) - \left(\lambda m(x) + \sum_{k=1}^{d-1}\lambda^k \int_{x}^{1}G^k_m(y)dm(y)\right)\right) = 0. \label{}
		\end{align}
		We call this boundary condition the modified Neumann boundary condition.
	\end{enumerate}
 \end{Prop}

 Since $\varphi^d_m(\lambda;\cdot)$ and $\psi_m(\lambda;\cdot)$ are linearly independent solutions of the equation $\frac{d}{dm}\frac{d^+}{dx}u = \lambda u$ for $d \geq d(m)$, the function $g_m(\lambda;\cdot)$ can be represented as 
 \begin{align}
 g_m(\lambda;\cdot) = \varphi^d_m(\lambda;\cdot) - c^d_m(\lambda)\psi_m(\lambda;\cdot) \label{defofg}
 \end{align} 
 by a constant $c^d_\lambda(m)$. From a simple calculation, we can prove
 	\begin{align}
 	c^{d+1}_m(\lambda) = c^d_m(\lambda) - \lambda^{d+1} \int_{0}^{1}G^d_m(x)dm(x). \label{relofcd}
 	\end{align}
 	
 We prepare some estimates for $G^k_m$.
 
 \begin{Prop} \label{Iest}
 	Let $m \in \cM$ and $k\geq 1$. Then $(-1)^{k}G^{k}_m(x)$ is non-decreasing and concave on $[0,1]$. In particular, for $0 < x < y \leq 1$, it holds that
 	\begin{align}
 	\frac{(-1)^{k}G^{k}_m(y)}{y} \leq \frac{(-1)^{k}G^{k}_m(x)}{x}. \label{}
 	\end{align}
 \end{Prop}
 \begin{proof}
 	We prove by induction. First we prove the case: $k=1$.
 	Since $-\tilde{m}$ is non-increasing and non-negative on $(0,1]$, the function $-G^1_m$ is concave and non-decreasing on $[0,1]$. Next we assume that for an integer $k \in \bN$ the assertion holds. By the equality 
 	\begin{align}
 	(-1)^{k+1}G^{k+1}_m(x) = \int_{0}^{x}dy \int_{y}^{1}(-1)^{k}G^{k}_m(z)dm(z) \label{}
 	\end{align} and the induction hypothesis, the function $\int_{x}^{1}(-1)^{k}G^{k}_m(z)dm(z)$ is non-increasing for $x \in [0,1]$. Since $(-1)^kG^k_m(0) = 0$, the function $(-1)^{k+1}G^{k}_m$ is non-negative on $[0,1]$. Hence $(-1)^{k+1}G^{k+1}_m$ is concave and non-decreasing on $[0,1]$.
 \end{proof}
 \begin{Prop}\label{divofM}
 	Let $m \in \cM$. Then for every $1 \leq k \leq d(m)$,
 	\begin{align}
 	\lim_{x \to +0}\frac{(-1)^{k}G^{k}_m(x)}{x} = \infty. \label{}
 	\end{align}
 \end{Prop}
 \begin{proof}
 	Take $1 \leq k \leq d(m)$ and $x \in [0,1]$. From Fubini's theorem, we have 
 	\begin{align}
 	(-1)^{k}G^{k}_m(x) &= \int_{0}^{x}z(-1)^{k-1}G^{k-1}_m(z) dm(z) + x \int_{x}^{1}(-1)^{k-1}G^{k-1}_m(z) dm(z) \label{}\\
 	&\geq x \int_{x}^{1}(-1)^{k-1}G^{k-1}_m(z) dm(z). \label{}
 	\end{align}
 	Then it follows that 
 	\begin{align}
 	\liminf_{x \to 0}\frac{(-1)^kG^{k}_m(x)}{x} \geq \int_{0}^{1}(-1)^{k-1}G^{k-1}_m(z) dm(z) = \infty .\label{}
 	\end{align}
 \end{proof}
 
 \begin{Prop} \label{divofI}
 	Let $m \in \cM$. Then the following holds: 
 	\begin{enumerate}
 		\item For $k \geq 1$ and $0 \leq x \leq y \leq 1$,
 		\begin{align}
 		0 \leq (-1)^{k+1}G^{k+1}_m(x) \leq (-1)^{k}G^{k}_m(x)\int_{0}^{y}zdm(z) + x\int_{y}^{1}(-1)^{k}G^{k}_m(z) dm(z). \label{eq6}
 		\end{align}
 		In particular, 
 		\begin{align}
 		0 \leq (-1)^{k+1}G^{k+1}_m(x) \leq (-1)^{k}G^{k}_m(x)\int_{0}^{1}zdm(z). \label{eq181}
 		\end{align}
 		\item For $1 \leq k \leq d(m)$,
 		\begin{align}
 		\lim_{x \to 0} \frac{G^{k+1}_m(x)}{G^{k}_m(x)} = 0. \label{eq4}
 		\end{align}
 		In particular, for $1 \leq k \leq d(m)$,
 		\begin{align}
 		\lim_{x \to 0}\frac{G^{k+2}_m(x)}{G^2_m(x)}= 0. \label{eq188}
 		\end{align}
 	\end{enumerate}
 \end{Prop}
 \begin{proof}
 	(i) We have
 	\begin{align}
 	&(-1)^{k+1}G^{k+1}_m(x)  \label{} \\
 	=& \int_{0}^{x}(-1)^{k}zG^{k}_m(z) dm(z) + x\int_{x}^{y}(-1)^{k}G^{k}_m(z) dm(z) \label{} \\	
 	&+ x\int_{y}^{1}(-1)^{k}G^{k}_m(z) dm(z). \label{}
 	\end{align}
 	From Proposition \ref{Iest}, it holds
 	\begin{align}
 	\int_{0}^{x}(-1)^{k}zG^{k}_m(z) dm(z) \leq (-1)^kG^k_m(x)\int_{0}^{x}zdm(z)
 	\end{align}
 	and
 	\begin{align}
 	x\int_{x}^{y}(-1)^{k}G^{k}_m(z) dm(z) \leq (-1)^kG^k_m(x)\int_{x}^{y}zdm(z). \label{}
 	\end{align}
 	Hence we obtain \eqref{eq6}. \\
 	(ii) From (i), we have for $0 < x < y \leq 1$
 	\begin{align}
 	\frac{(-1)^{k+1}G^{k+1}_m(x)}{(-1)^kG^k_m(x)} \leq \int_{0}^{y}zdm(z) + \frac{x}{(-1)^kG^k_m(x)}\int_{y}^{1}(-1)^kG^k_m(z)dm(z). \label{}
 	\end{align}
 	Hence from Proposition \ref{divofM}, we have
 	\begin{align}
 	\limsup_{x \to +0}\frac{(-1)^{k+1}G^{k+1}_m(x)}{(-1)^kG^k_m(x)} \leq \int_{0}^{y}zdm(z). \label{}
 	\end{align}
 	Since $y > 0$ can be taken arbitrary small, we obtain the desired result.
 \end{proof}

\section{Convergence of inverse local times}\label{section: contthmofinverselocaltime}

Here we briefly review the construction of sample paths of jumping-in diffusions. 
Let us assume a strong Markov process $X$ on $[0,\infty)$ has continuous paths and natural scale on $(0,\infty)$ and, as soon as $X$ hits the origin, it jumps into $(0,\infty)$. 
Then there exist Radon measures $m$ on $(0,\infty)$ with full support and $j$ on $(0,\infty)$ and such a process has the local generator $\cL$ of the form
\begin{equation}
\cL u(x) = \frac{d}{dm}\frac{d^+}{dx}u(x) \quad \text{for} \ x \in (0,\infty)  \label{}
\end{equation} 
subject to Feller's boundary condition (see e.g., Feller \cite{Feller:Theparabolic})
\begin{align}
\int_{0}^{\infty}(u(x) - u(0))j(dx) = 0. \label{}
\end{align}
Conversely, we can construct such processes from $(m,j)$ with
\begin{equation}
\text{(C)}
\left\{
\begin{aligned}
&(i)\  j(1, \infty) + \int_{0}^{1}xj(dx) + \int_{0}^{1}|G_m(x)|j(dx) < \infty,  \\
&(ii)\  j(0,1) = \infty
\end{aligned}
\label{existenceoftheprocess2}
\right.
\end{equation}
(see It\^o \cite{Ito:PPP} and Rogers \cite{Rogers1}).

Sample paths of a unilateral jumping-in diffusion is constructed from $m$ and $j$ via It\^o's excursion theory as follows: Let $m$ and $j$ be Radon measures on $(0,\infty)$ with the condition (C).
Define $T_0(e) = \inf\{ s > 0 \mid e(s) = 0 \} \ (e \in \bD)$ and $\bE$ as the set of all elements $e$ in $\bD$ which satisfy that $e(u) = 0$ for every $u \geq T_0(e)$ if $T_0(e) < \infty$.
Let $P^m_x \ (x \in (0,\infty))$ be the law of the diffusion process with speed measure $dm$ starting from $x$ and killed at $0$.
Then the following measure $n_{m,j}$ is the excursion measure of the process $X_{m,j}$:
\begin{align}
n_{m,j}(A) = \int_{0}^{\infty}P^{m}_x (A)j(dx) \quad (A \in \cB(\bE)). \label{excursion_measure} 
\end{align} 
We construct the sample paths of $X_{m,j}$.
We define $N(ds de)$ as a Poisson random measure on $(0,\infty)\times \bE$ having intensity measure $dx \otimes n_{m,j}$ being defined on some probability space. Here $dx$ is the Lebesgue measure. We define $D(p) = \{ s \in (0,\infty) \mid N(\{s\}\times \bE ) = 1 \}$ and the map $p:D(p) \to \bE$ 
such that $p[s] \ (s \in D(p))$ is the only one element of the support of the measure $N(\{s\} \times de)$.
We define the process $\eta_{m,j}$ as follows:
\begin{align}
\eta_{m,j}(u) = \int_{(0,u] \times \bE}T_0(e) N(ds de). \label{}
\end{align}
Then we construct $X_{m,j}$ as follows: 
\begin{equation}
X_{m,j}(t) = \left\{
\begin{aligned}
&p[u](t - \eta_{m,j}(u-)) & &\text{if} \ u \in D(p) \ \text{and} \ \eta_{m,j}(u-) \leq t < \eta_{m,j}(u),  \\
&0 & &\text{otherwise}.
\end{aligned}
\right. \label{}
\end{equation}
Then $\eta_{m,j}$ plays the role of the inverse local time at $0$ of $X_{m,j}$.
This process $X_{m,j}$ thus constructed is the unilateral jumping-in diffusion associated to $m$ and $j$ and started from $0$. We denote the law of $X_{m,j}$ by $P$.

We explain a construction of a bilateral jumping-in diffusion.
Let $m_+,m_-$ and $j_+,j_-$ be Radon measures on $(0,\infty)$ and suppose $(m_+,j_+)$ and $(m_-,j_-)$ satisfy (C). Then we construct the bilateral jumping-in diffusion processes $X_{m_+,j_+;m_-,j_-}$ i.e., a Markov process on $\bR$ which behaves like $X_{m_+,j_+}$ while $X$ is positive and like $-X_{m_-,j_-}$ while $X$ is negative and as soon as the process hits the origin it jumps into $\bR \setminus \{ 0 \}$ according to $j_+$ and $j_-$.
The precise definition is as follows. 
Take two independent Poisson random measures $N_+$ and $\tilde{N}_-$ defined on a common probability space whose intensity measures are $n_{m_+,j_+}$ and $n_{m_-,j_-}$, respectively. 
Define $N_{m_-,j_-}(ds de) = \tilde{N}_{m_-,j_-}(ds d(-e))$. Then we can define $X_{m_+,j_+;m_-,j_-}$ from the excursion point process $N_{m_+,j_+} + N_{m_-,j_-}$.

Here we show a useful expression for the Laplace exponent of $\eta_{m,j}$.
For $\lambda > 0$, we have
\begin{align}
\chi_{m,j}(\lambda) &:= -\log P[\mathrm{e}^{-\lambda \eta_{m,j}(1)}] \label{} \\
&\ = \int_{0}^{\infty}(1 - \mathrm{e}^{-\lambda u}) n_{m,j}(T_0 \in du)  \label{} \\
&\ = \int_{0}^{\infty}P^m_x[ 1 - \mathrm{e}^{-\lambda T_0}]j(dx). \label{}
\end{align} 
It is well-known that the following holds (see e.g., \cite{Ito:Essentialsof}):
\begin{align}
g_m(\lambda;x) = P^m_x[\mathrm{e}^{-\lambda T_0}]. \label{}
\end{align}
Hence we obtain the following representation:
\begin{align}
\chi_{m,j}(\lambda) &= \int_{0}^{\infty}(1 - g_m(\lambda;x))j(dx) \quad (\lambda > 0 ). \label{eq94}
\end{align}

Before we state our continuity theorem, we prepare estimates of $G^k_{m_n}$ for a sequence of strings $\{m_n\}_n$.
\begin{Prop}\label{elemestofG2}
	Let $m_n \in \cM$. Assume the following holds:
	\begin{align}
	&\lim_{n \to \infty}m_n(x) = 0 \ \text{for every} \ x > 0. \label{} \\
	&\lim_{n \to \infty}\int_{0}^{1}ydm_n(y) = 0. \label{}
	\end{align}
	Then for every $0 < x \leq 1$ and $k \geq 1$, we have the following:
	\begin{enumerate}
		\item $\lim_{n \to \infty}\sup_{z \in [0,x_0]}|G^k_{m_n}(z)| = 0$ \quad for every $x_0 > 0$.
		\item $\lim_{n \to \infty}\int_{x_0}^{x_1}|G^k_{m_n}(z)|dm_n(z) = 0$ \quad for every $0 < x_0 < x_1$.
		\item $\lim_{n \to \infty} \sup_{y \in (0,1]}\left| \frac{G^{k+1}_{m_n}(y)}{G^k_{m_n}(y)} \right| = 0$ \quad for every $k \geq 1$.
	\end{enumerate}
 \end{Prop}
 \begin{proof}
	(i) Fix $x_0 \geq 1$. We prove by induction that the following holds: 
	\begin{align}
	\lim_{n \to \infty}\sup_{y \in [0,x_0]}|G^k_{m_n}(y)| = 0 \ \text{for every} \ k \geq 1. \label{eq67} 
	\end{align}
	First we show the case: $k=1$.
	Since it follows that
	\begin{align}
		\sup_{x \in [0,x_0]}|G^1_{m_n}(x)| \leq \int_{0}^{1}ydm_n(y) + \int_{1}^{x_0}|\tilde{m}_n(y)|dy, \label{}
	\end{align}
	we obtain \eqref{eq67} for $k=1$. Next we assume that \eqref{eq67} holds for an integer $k \geq 1$. From \eqref{eq181}, for $0 < x \leq 1$, it follows
	\begin{align}
	(-1)^{k+1}G^{k+1}_{m_n}(x) \leq (-1)^kG^k_{m_n}(x)\int_{0}^{1}ydm_n(y). \label{eq182}
	\end{align}
	Then it holds that
	\begin{align}
		\lim_{n \to \infty}\sup_{x \in [0,1]}|G^{k+1}_{m_n}(x)| = 0 \label{}.
	\end{align}
	For $x \in [1,x_0]$, we have
	\begin{align}
		|G^{k+1}_{m_n}(x)| &= \left| -x\int_{1}^{x}(-1)^kG^k_{m_n}(y)dm_n(y)
		+\int_{0}^{x}(-1)^kyG^k_{m_n}(y)dm_n(y) \right| \label{} \\
		&\leq \sup_{z \in [0,x_0]}|G^k_{m_n}(z)|\left(x_0|\tilde{m}_n(x_0)| +  \int_{0}^{x_0}ydm_n(y)\right). \label{}
	\end{align}
	Hence we obtain \eqref{eq67} for $k+1$, and therefore we obtain (i). \\
	(ii) This is obvious from (i) and the dominated convergence theorem. \\
	(iii) This is obvious \eqref{eq182}.
 \end{proof}
We introduce notion of convergence for strings $m$ with $d(m) < \infty$.
\begin{Def} \label{convofstring}
 	For $m_n \in \cM$, we denote $m_n \xrightarrow{G} 0$ when the following holds:
 	\begin{enumerate}
 		\item $\lim_{n \to \infty}m_n(x) = 0$ for every $x > 0$.
		\item $\lim_{n \to \infty}\int_{0}^{1}ydm_n(y) = 0$.
 		\item $\lim_{n \to \infty}\int_{0}^{1}(-1)^dG^d_{m_n}(x)dm_n(x) = 0$ for some integer $d \geq 1$.
 	\end{enumerate}
 \end{Def}
\begin{Rem}
	Note that from the condition (iii), we have $\limsup_{n \to \infty} d(m_n) \leq d < \infty$.
\end{Rem}

 Now we establish a continuity theorem for inverse local times $\eta_{m,j}$.  
 \begin{Thm}\label{convtoBM2}
 	Let $m_n \in \cM$ and $j_n$ be a Radon measure on $(0,\infty)$. Suppose the following holds: 
 	\begin{enumerate}
		\item $m_n \xrightarrow{G} 0$.
 		\item $\limsup_{n \to \infty}\int_{0}^{1}xj_n(dx) < \infty$.
 		\item $j_n(dx) \xrightarrow[n \to \infty]{w} 0$ on $[1,\infty]$.
 		\item $G^2_{m_n}(x)j_n(dx) \xrightarrow[n \to \infty]{w} \kappa\delta_0(dx) $ on $[0,1]$ for a constant $\kappa \geq 0$.
 	\end{enumerate}
 	Then if we take $b_n = -\int_{0}^{1}G_{m_n}(x)j_n(dx)$, we have 
 	\begin{align}
 	\eta_{m_n,j_n}(t) - b_n t \xrightarrow[n \to \infty]{d} B(2\kappa t) \quad \text{on} \ \bD. \label{}
 	\end{align}
 \end{Thm}
  For the proof, we need two propositions. One of them can be seen as a kind of continuity of $c^d_{m_n}(\lambda)$ with respect to $m_n$ under the convergence $m_n \xrightarrow{G} 0$.
 \begin{Prop}\label{vanishofcd}
 	Let $m_n \in \cM$. Suppose $m_n \xrightarrow{G} 0$. Then the following holds for an integer $d \geq 1$:
 	\begin{align}
 	\lim_{n \to \infty}c^d_{m_n}(\lambda) = 0. \label{}
 	\end{align}
 \end{Prop}
 \begin{proof}
 	Fix an integer $d \geq 1$ so that the following holds:
 	\begin{align}
 	\lim_{n \to \infty}\int_{0}^{1}G^d_{m_n}(x)dm_n(dx) = 0. \label{}
 	\end{align}
 	Since it holds that
 	\begin{align}
 	g_{m_n}(\lambda;x) = \varphi^d_{m_n}(\lambda;x) - c^d_{m_n}(\lambda)\psi_{m_n}(\lambda;x) \quad (x \in [0,1]), \label{}
 	\end{align}
 	it is enough to show 
 	\begin{align}
 	\lim_{n \to \infty}\varphi^d_{m_n}(\lambda;x) = 1,\ \lim_{n \to \infty}\psi_{m_n}(\lambda;x) = x,\ \lim_{n \to \infty}g_{m_n}(\lambda;x) = 1 \label{} 
 	\end{align}
 	for some $x > 0$. Fix $x \in (0,1]$.
 	From \eqref{eq122}, we have
 	\begin{align}
 	0 \leq \psi_{m_n}(\lambda;x) - x \leq x\lambda (m_n \bullet s)(x)\mathrm{e}^{\lambda (m_n \bullet s)(x)}. \label{}
 	\end{align}
 	Since $m_n \xrightarrow{G} 0$ holds, we have
 	\begin{align}
 	\lim_{n \to \infty}\psi_{m_n}(\lambda;x) = x \quad (x \geq 0, \lambda > 0). \label{}
 	\end{align}
 	Then from \eqref{eq8}, we also obtain 
 	\begin{align}
 	\lim_{n \to \infty}g_{m_n}(\lambda;x) = 1. \label{}
 	\end{align}
 	From \eqref{eq183}, we have
 	\begin{align}
 	\left|\varphi^d_{m_n}(\lambda;x) - 1 - \sum_{k=1}^{d}\lambda^kG^k_{m_n}(x) \right|
 	\leq (-1)^{d}\lambda^{d+1}\left(\int_{0}^{x}G^{d}_{m_n}(y)dm_n(y)\right)x\mathrm{e}^{\lambda (m_n \bullet s)(x)}, \label{}
 	\end{align}
 	which converges to $0$ as $n \to \infty$.
 	From (i) of Proposition \ref{elemestofG2}, we have
 	\begin{align}
 	\lim_{n \to \infty}G^k_{m_n}(x) = 0 \label{}
 	\end{align}
 	for $1 \leq k \leq d$. Then it follows that
 	\begin{align}
 	\lim_{n \to \infty}\varphi^d_{m_n}(\lambda;x) = 1. \label{}
 	\end{align}
 	The proof is complete.
 \end{proof}
 The other proposition is a continuity theorem for Laplace transforms of L\'evy processes without negative jumps. 
 \begin{Prop}\label{continuityofLT}
	Let $X_n,X$ be real-valued random variables. Assume
	\begin{align}
		P[\mathrm{e}^{-\lambda X_n}], \quad P[\mathrm{e}^{-\lambda X}] < \infty
		\quad \text{and} \quad \lim_{n \to \infty}P[\mathrm{e}^{-\lambda X_n}] = P[\mathrm{e}^{-\lambda X}] \quad (\lambda > 0). \label{}
	\end{align}
	Then we have
	\begin{align}
		X_n \xrightarrow[n \to \infty]{d} X. \label{}
	\end{align}
 \end{Prop}
 We omit the proof of Proposition \ref{continuityofLT} (see, e.g., \cite[Proposition A.1]{YamatoYano}).

 Now we prove Theorem \ref{convtoBM2}.

 \begin{proof}[Proof of Theorem \ref{convtoBM2}]
 	Take an integer $d$ such that
 	\begin{align}
 		\limsup_{n \to \infty}\int_{0}^{1}(-1)^dG^d_{m_n}(x)dm_n(x) < \infty. \label{}
 	\end{align}
	From \eqref{eq94} and Proposition \ref{continuityofLT}, it is enough to show the following for every $\lambda > 0$:
	\begin{align}
	\lim_{n \to \infty}\left( \int_{0}^{\infty}(1-g_{m_n}(\lambda;x))j_n(dx) + \lambda \int_{0}^{1}G_{m_n}(x)j_n(dx)\right) = -\kappa\lambda^2. \label{eq184}
	\end{align}
	Fix $\lambda > 0$.
	Since the function $1 - g_{m_n}$ is bounded continuous,	we see from the assumption (iii) that \eqref{eq184} is equivalent to the following (with $(\lambda)$ being omitted):
	\begin{align}
	\lim_{n \to \infty} \int_{0}^{1}(1-g_{m_n} + \lambda G_{m_n})dj_n = -\kappa\lambda^2. \label{eq187}
	\end{align}
	From \eqref{defofg}, we have
	\begin{align}
	1 - g_{m_n} + \lambda G_{m_n} = 1 - (\varphi^d_{m_n} - c^d_{m_n}\psi_{m_n}) + \lambda G_{m_n} 
	= - \Phi^d_{m_n} + c^d_{m_n}\psi_{m_n} -\sum_{k=1}^{d}\lambda^k G^k_{m_n}, \label{}
	\end{align}
	where $\Phi^d_{m_n} = \varphi^d_{m_n} - 1 - \sum_{k=1}^{d}\lambda^k G^k_{m_n}$.
	From \eqref{eq183}, we have for $x \in [0,1]$
	\begin{align}
		|\Phi^d_{m_n}| \leq	x\lambda^{d+1}\mathrm{e}^{(m_n \bullet s)(x)}\int_{0}^{x} (-1)^d G^d_{m_n}(y)dm_n(y). \label{}		
	\end{align}
	Then from the assumption (i) and (ii), it holds
	\begin{align}
		\lim_{n \to \infty}\int_{0}^{1}\Phi^d_{m_n}dj_n = 0. \label{}
	\end{align}
	Similarly, from \eqref{eq122}, we can show
	\begin{align}
	\limsup_{n \to \infty}\int_{0}^{1}\psi_{m_n}dj_n < \infty, \label{eq186}
	\end{align}
	and we have from Proposition \ref{vanishofcd} 
	\begin{align}
		\lim_{n \to \infty}c^d_{m_n} = 0. \label{}
	\end{align}
	From the assumption (iv) and \eqref{eq188}, it follows that
	\begin{align}
		\lim_{n \to \infty}\int_{0}^{1}\sum_{k=1}^{d}\lambda^k G^k_{m_n}dj_n 
		= \lim_{n \to \infty}\lambda^2\int_{0}^{1}G^2_{m_n}dj_n = \kappa\lambda^2. \label{}
	\end{align}
	The proof is complete.
\end{proof}

 \section{Scaling limit of inverse local times}\label{section: scalinglimit}
 Applying the continuity theorems shown in the previous section, we establish the scaling limit of $\eta_{m,j}$.  
 We reduce the scaling limit of $\eta_{m,j}$ to the continuity of it with respect to $m$ and $j$ by the change of variables.
 For every $m \in \cM$, $a,b > 0$ and $x > 0$, it may not be difficult to see that 
 \begin{align}
 g_{m}(a\lambda;bx) &= g_{am^b}(b\lambda;x), \label{} \\
 \psi_{m}(a\lambda;bx) &= b\psi_{am^b}(b\lambda;x) \label{}
 \end{align}
 for $m^b(x) = m(bx)$.

 Now we prove Theorem \ref{informal-main-alpha2+}.
 \begin{proof}[Proof of Theorem \ref{informal-main-alpha2+}]
 	We may assume $m(\infty) = 0$ without loss of generality.
 	Define
 	\begin{align}
 		m_\gamma(x) = \frac{\gamma^{\alpha/2 - 1/2}}{u(\gamma^{\alpha/2})}m(\gamma^{\alpha/2} x), \quad  j_\gamma(dx) = \frac{\gamma}{v(\gamma^{\alpha/2})}j(d(\gamma^{\alpha/2} x)). \label{eq225}
 	\end{align}
 	From the change of variables, we have for $\gamma' := \gamma^{\alpha / 2}$
 	\begin{align}
 		\frac{1}{\gamma^{1/2}u(\gamma')}\left(\eta_{m,j}\left(\frac{\gamma t} {v(\gamma')} \right) - b\frac{\gamma t}{v(\gamma')}\right)
 		\overset{d}{=} \eta_{m_\gamma,j_\gamma}(t) - b_\gamma t \quad \text{on} \ \bD, \label{}
 	\end{align}
 	where $b_\gamma = -\int_{0}^{\infty}G_{m_\gamma}(x)j_\gamma (dx)$.
 	Then it is enough to show that $\{ m_\gamma \}_\gamma$ and $\{ j_\gamma \}_\gamma$ satisfy the conditions of Theorem \ref{convtoBM2} for $\kappa = 1$ and
 	\begin{align}
 		&\lim_{\gamma \to \infty}\int_{1}^{\infty}G_{m_\gamma} (x)j_\gamma(dx) = 0. \label{eq175}
 	\end{align}
 	Since the conditions (ii) and (iii) of Theorem \ref{convtoBM2} are almost obvious from the assumptions on the function $v$, we omit the proof.
 
 	First we show $m_\gamma \overset{G}{\to} 0 \quad (\gamma \to \infty)$. 
	Since the conditions (i) and (ii) in Definition \ref{convofstring} are obvious from the definition of the function $u$, we only prove
 	\begin{align}
 		\lim_{\gamma \to \infty}\int_{0}^{1}G^{d}_{m_\gamma}(x)dm_\gamma(x) = 0 \label{eq216}
 	\end{align}
 	for some $d$.
	This needs some lengthy computations.
 	Set $N := \max \{k \in \bN \mid k < \alpha \}$ and fix an integer $d \geq d(m) \vee N$.
 	Set
 	\begin{align}
 		H^k(\gamma;x) =\left\{ 
 		\begin{aligned}
 			&-G_m(x) & (k = 1), \\
 			&\int_{0}^{x}dy\int_{y}^{\infty}H^{k-1}(\infty;z)dm(z) & (2 \leq k \leq N), \\
 			&\int_{0}^{x}dy\int_{y}^{\gamma}H^{k-1}(\gamma;z)dm(z) & (k > N)
 		\end{aligned} \label{H-func}
 		\right.
 	\end{align}
 	for $x > 0$ and $\gamma > 0$.
 	We may easily see that $\int_{0}^{1}H^{l}(\gamma;x)dm(x) < \infty$ for $l \leq d(m)$.
 	Note that for $k \leq N$ it holds
 	\begin{align}
 		\int_{1}^{\infty}H^{k}(\infty;x)dm(x) < \infty. \label{eq214}
 	\end{align}
 	Indeed, from induction and Karamata's theorem \cite[Proposition 1.5.8]{Regularvariation} it holds
 	\begin{align}
 		H^{k}(\infty;x) \sim A_k x^{k/\alpha}K(x)^{k}  \quad  (k \leq N) \label{eq215}
 	\end{align}
 	for a constant $A_k > 0$. Thus \eqref{eq214} holds.
 	Note also that $H^k(\gamma;x) \geq 0$ and non-decreasing for $0 < x \leq \gamma$.
 	Since it holds $-G^1_m(x) \leq -G_m(x)$, by changing variables, it follows for $0 \leq x \leq 1$
 	\begin{align}
 		\int_{x}^{1}G^2_{m_\gamma}(y)dm_{\gamma}(y) &\leq \int_{x}^{1}dm_{\gamma}(y)\int_{0}^{y}dz\int_{z}^{1}(-G_{m_\gamma}(w))dm_{\gamma}(w) \label{} \\
 		&\leq \frac{1}{u(\gamma^{\alpha/2})^{3}\gamma^{(3-\alpha)/2}} \int_{\gamma^{\alpha/2} x}^{\gamma^{\alpha/2}}H^2(\gamma^{\alpha/2}; y)dm(y). \label{}
 	\end{align}
 	Similarly, it inductively holds for $k \geq 1$,
 	\begin{align}
 		(-1)^{k}\int_{x}^{1}G^k_{m_\gamma}(y)dm_{\gamma}(y)
 		\leq \frac{1}{u(\gamma^{\alpha/2})^{k+1}\gamma^{(k+1-\alpha)/2}}\int_{\gamma^{\alpha/2} x}^{\gamma^{\alpha/2}}H^k(\gamma^{\alpha/2};y)dm(y) \quad (0 < x \leq 1). \label{eq212}
 	\end{align}
	 Thus we estimate $H^k \ (k \geq 1)$.
	Set $F(\gamma) := \int_{0}^{\gamma}xdm(x)$.
 	We show by induction that for every $k \geq N$ there exists a constant $C_k \geq 0$, and it holds
 	\begin{align}
 		H^{k}(\gamma;x) \leq H^{k}(1;x) + C_k  F(\gamma)^{k-N-1}M(\gamma)x \quad ( \gamma \geq 1,0 < x \leq \gamma) \label{eq210}
 	\end{align}
 	and
 	\begin{align}
 		\left|\int_{1}^{\gamma}H^k(1;x)dm(x)\right| \leq C_k F(\gamma)^{k-N}M(\gamma) \quad ( \gamma \geq 1), \label{eq213}
 	\end{align}
 	where 
	\begin{align}
		M(\gamma) := \int_{1}^{\gamma}H^N(\infty;x)dm(x). \label{eq226}
	\end{align}
 	The case $k=N$ is obvious since $H^N(\gamma;x)$ does not depend on $\gamma$.
 	Assume \eqref{eq210} and \eqref{eq213} hold for $k \geq N$.
 	Then it follows
 	\begin{align}
 		&H^{k+1}(\gamma;x) \label{} \\
 		= &\int_{0}^{x}dy\int_{y}^{\gamma}H^{k}(\gamma;z)dm(z) \label{} \\
 		\leq &\int_{0}^{x}dy\int_{y}^{\gamma}(H^{k}(1;z) + C_k F(\gamma)^{k-N-1}M(\gamma) z)dm(z) \label{} \\
 		\leq &\int_{0}^{x}dy\int_{y}^{1}H^{k}(1;z)dm(z) + x\int_{1}^{\gamma}H^{k}(1;z)dm(z) + C_k F(\gamma)^{k-N}M(\gamma) x. \label{} \\
 		\leq &H^{k+1}(1;x) + 2C_k F(\gamma)^{k-N}M(\gamma)x, \label{}
 	\end{align}
 	and
 	\begin{align}
 		&\left|\int_{1}^{\gamma}H^{k+1}(1;x) dm(x)\right|  \label{} \\
 		= &\left|\int_{1}^{\gamma}dm(x)\int_{0}^{x}dy\int_{y}^{1}H^{k}(1;z) dm(z) \right| \label{} \\
 		\leq &\left| \int_{1}^{\gamma}dm(x)\int_{0}^{1}dy\int_{y}^{1}H^{k}(1;z) dm(z) \right| + \int_{1}^{\gamma}dm(x)\int_{1}^{x}dy\left|\int_{1}^{\gamma}H^{k}(1;z) dm(z) \right| \label{} \\
 		\leq &-m(1)H^{k+1}(1;1) + F(\gamma) \left|\int_{1}^{\gamma}H^k(1;x)dm(x)\right| \label{} \\
 		\leq &-m(1)H^{k+1}(1;1) + C_kF(\gamma)^{k+1-N}M(\gamma). \label{}
 	\end{align}
 	Thus \eqref{eq210} and \eqref{eq213} hold for every $k \geq N$.
 	From \eqref{eq215}, it holds that
 	\begin{align}
 		M(\gamma) = O(\gamma^{(N+1-\alpha)/\alpha}K(\gamma)^{N+1}) \quad (\gamma \to \infty). \label{eq220}
 	\end{align}
 	From \eqref{eq212}, \eqref{eq210} and \eqref{eq213}, we have
 	\begin{align}
 		&u(\gamma^{\alpha/2})^{d+1}\gamma^{(d+1-\alpha)/2}(-1)^{d}\int_{0}^{1}G^d_{m_\gamma}(y)dm_{\gamma}(y) \label{} \\
 		\leq& \int_{0}^{\gamma^{\alpha/2}}H^d(\gamma^{\alpha/2};y)dm(y) \label{} \\
 		\leq& \int_{0}^{\gamma^{\alpha/2}}( H^d(1;y) + C_dF(\gamma^{2/\alpha})^{d-N-1}M(\gamma^{\alpha/2})y  ) dm(y) \label{} \\
 		=& \int_{0}^{1}H^d(1;y)dm(y) +  2C_d F(\gamma^{\alpha/2})^{d-N}M(\gamma^{\alpha/2}) \label{} \\
 		=& O(F(\gamma^{\alpha/2})^{d-N}M(\gamma^{\alpha/2}))  \label{} \\
 		=& O(\gamma^{(d-N)/2}K(\gamma^{\alpha/2})^{d-N}M(\gamma^{\alpha/2})),  \label{eq227}
 	\end{align}
 	where we used that $F(\gamma^{\alpha/2}) = O(\gamma^{1/2}K(\gamma^{\alpha/2})) \quad (\gamma \to \infty)$, which follows from Karamata's theorem.
 	Thus from \eqref{eq220} it follows that
 	\begin{align}
 		 \int_{0}^{1}G^d_{m_\gamma}(y)dm_{\gamma}(y) = O\left(\frac{K(\gamma^{\alpha/2})^{d+1}}{u(\gamma^{\alpha/2})^{d+1}} \right) \xrightarrow{\gamma \to \infty} 0. \label{}
 	\end{align}
 	Thus \eqref{eq216} holds, and we obtain $m_\gamma \overset{G}{\to} 0$.

	Next we check \eqref{eq175}.
	By changing variables and Karamata's theorem, we have
	\begin{align}
		-\int_{1}^{\infty}G_{m_\gamma}(x)j_\gamma(dx)
		&= -\frac{\gamma^{1/2}}{u(\gamma^{\alpha/2})v(\gamma^{\alpha/2})}\int_{\gamma^{\alpha/2}}^{\infty}G_m(x)j(dx) \label{} \\
		&= -\frac{\gamma^{1/2}}{u(\gamma^{\alpha/2})v(\gamma^{\alpha/2})}\left( G_m(\gamma^{\alpha/2})j(\gamma^{\alpha/2}, \infty) + \int_{\gamma^{\alpha/2}}^{\infty}m(x)j(x,\infty)dx\right) \label{} \\
		&\sim 2\alpha\left(\frac{\gamma^{(\alpha-1)/2} m(\gamma^{\alpha/2})}{u(\gamma^{\alpha/2})}\right)\left(\frac{\gamma j(\gamma^{\alpha/2},\infty)}{v(\gamma^{\alpha/2})}\right) \quad (\gamma \to \infty). \label{}
	\end{align}

	We next show the condition (iv) of Theorem \ref{convtoBM2}, that is,
 	\begin{align}
 		\lim_{\gamma \to \infty}\int_{0}^{1}f(x)G^2_{m_\gamma}(x)j_\gamma(dx) = f(0) \label{eq204}
 	\end{align}
 	for every bounded continuous function $f: [0, 1] \to \bR$.
 	By the change of variables, we have
 	\begin{align}
 		&\int_{0}^{1}f(x)G^2_{m_\gamma}(x)j_\gamma(dx) \label{} \\
 		=& \frac{-1}{u(\gamma')^2v(\gamma')}\int_{0}^{\gamma'}f(\gamma'^{-1}x)j(dx) \int_{0}^{x}dy\int_{y}^{\gamma'}G_m(z)dm(z) \label{} \\
 		&+ \frac{m(\gamma')}{u(\gamma')^2v(\gamma')} \int_{0}^{\gamma'}f(\gamma'^{-1}x)j(dx)\int_{0}^{x}dy\int_{y}^{\gamma'}zdm(z) . \label{}
 	\end{align}
 	Since it holds
 	\begin{align}
 		\int_{0}^{\gamma'}j(dx)\int_{0}^{x}dy\int_{y}^{\gamma'}zdm(z) \leq -G_m(\gamma')\int_{0}^{\gamma'}xj(dx), \label{}
 	\end{align}
 	we have from Karamata's theorem 
 	\begin{align}
 		\lim_{\gamma \to \infty}\frac{m(\gamma')}{u(\gamma')^2v(\gamma')} \int_{0}^{\gamma'}f(\gamma'^{-1}x)j(dx)\int_{0}^{x}dy\int_{y}^{\gamma'}zdm(z) = 0. \label{}
 	\end{align}
 	For $\delta \in (0,1)$, we have
 	\begin{align}
 		&\frac{-1}{u(\gamma')^2v(\gamma')}\int_{\delta\gamma'}^{\gamma'}f(\gamma'^{-1}x)j(dx) \int_{0}^{x}dy\int_{y}^{\gamma'}G_m(z)dm(z) \label{} \\
 		\leq & ||f||_{\infty}\frac{G_m(\gamma')^2j(\delta\gamma',\gamma')}{u(\gamma')^2v(\gamma')} \xrightarrow{\gamma \to \infty} 0, \label{}
  	\end{align}
  	where $||f||_\infty = \sup_{x \in[0,1]}|f(x)|$.
 	Thus it follows that
 	\begin{align}
 		\lim_{\gamma \to \infty}\int_{0}^{1}f(x)G^2_{m_\gamma}(x)j_\gamma(dx) 
 		\sim f(0) \frac{N(\gamma')}{u(\gamma')^2 v(\gamma')} \sim f(0) \quad (\gamma \to \infty), \label{}
 	\end{align}
 	and \eqref{eq204} holds.
	Now we have checked all the assumptions of Theorem \ref{convtoBM2}, we obtain the convergence \eqref{eq218}.

 	The asymptotic relation \eqref{eq217} is no more than the application of a Tauberian theorem.
 	Define 
 	\begin{align}
 		\nu(ds) :=  \left(\int_{s}^{\infty}n_{m,j}[T_0 > u]du \right)ds \quad \text{and} \quad 
		\hat{\nu}(\lambda) := \int_{0}^{\infty}\mathrm{e}^{-\lambda s}\nu(ds) = - \frac{\tilde{\chi}(\lambda)}{\lambda^2}, \label{}
 	\end{align}
 	where $\tilde{\chi}(\lambda) = \chi_{m,j}(\lambda) -b\lambda$.
 	Since we have already shown \eqref{eq218}, it holds that
 	\begin{align}
 		\tilde{\chi}(\lambda) \sim -\lambda^2U^{\sharp}(\lambda^{-2})^{-1}v(\lambda^{-\alpha}U^{\sharp}(\lambda^{-2})^{\alpha/2}) \quad (\lambda \to +0 ). \label{}
 	\end{align}
 	From Karamata's Tauberian theorem \cite[Theorem 1.7.1]{Regularvariation}, it follows
 	\begin{align}
 		\nu(0,s] \sim U^{\sharp}(s^2)^{-1}v(s^{\alpha}U^{\sharp}(s^2)^{\alpha/2}) \quad (s \to \infty ). 
 	\end{align}
 	Then applying the monotone density theorem \cite[Theorem 1.7.2]{Regularvariation} twice, we obtain \eqref{eq217}.

 \end{proof}

 To apply the Theorem \ref{convtoBM2}, we need to choose appropriate slowly varying functions $u$ and $v$ satisfying the assumptions.
 The following proposition gives the asymptotic behavior of the function $N$ in \eqref{eq197} in terms of the functions $K$ and $L$,
 from which we may know how we should take $u$ and $v$.
 
 \begin{Prop} \label{comparison of N,K,L alpha > 2}
	Let $m \in \cM$ with $d(m) < \infty$ and $j$ be a Radon measure on $(0,\infty)$.
    Assume $(m,j)$ satisfies $\mathrm{(C)}$.
    For functions $K$ and $L$ which are slowly varying at $\infty$ and $\alpha \geq 2$, suppose the following: 
	\begin{enumerate}
		\item $\mathrm{(M)}_{\alpha,K}$ holds.
		\item $\mathrm{(J)}_{2/\alpha,L}$ holds.
	\end{enumerate}
    Assume for the function $N$ defined by \eqref{eq197} that $N(\infty) = \infty$.
    Then we have
	\begin{align}
	   N(\gamma) \sim \left\{
	   \begin{aligned}
		   &\frac{\alpha}{(\alpha - 1)(\alpha -2)} \int_{1}^{\gamma}\frac{K(x)^2L(x)}{x}dx & (\alpha > 2), \\
		   &\int_{1}^{\gamma} \frac{K(x)^2}{x} dx\int_{1}^{x}\frac{L(y)}{y}dy & (\alpha = 2),
	   \end{aligned}
	   \right.
	   \quad (\gamma \to \infty). \label{eq219}
	\end{align}
\end{Prop}

\begin{proof}
	We may assume $m(\infty) = 0$ without loss of generality.
	By integration by parts and Fubini's theorem, we have
	\begin{align}
		&N(\gamma) \label{} \\
	   =& -\int_{0}^{\gamma}j(x,\gamma)dy\int_{y}^{\gamma}G_m(z)dm(z) \label{} \\ 
	   =& -\int_{0}^{\gamma}j(x,\gamma)\left( m(\gamma)G_m(\gamma) - m(x)G_m(x) - \int_{x}^{\gamma}m(y)^2dy \right)dx \label{} \\
	   =& -\int_{0}^{\gamma}j(x,\infty)\left( m(\gamma)G_m(\gamma) - m(x)G_m(x) - \int_{x}^{\gamma}m(y)^2dy \right)dx \label{} \\
	   &+ j(\gamma,\infty)\left( \gamma m(\gamma)G_m(\gamma) -(1/2)G_m(\gamma)^2 - \int_{0}^{\gamma}xm(x)^2dx \right) \label{} \\
	   &= I_1(\gamma) + I_2(\gamma) + I_3(\gamma) + I_4(\gamma), \label{}
	\end{align}
	where
	\begin{align}
	   I_1(\gamma) &= \int_{0}^{\gamma}m(x)^2dx \int_{0}^{x}j(y,\infty)dy, \quad 
	   I_2(\gamma) = \int_{0}^{\gamma} m(x) G_m(x) j(x,\infty) dx, \label{} \\
	   I_3(\gamma) &= - m(\gamma)G_m(\gamma) \int_{0}^{\gamma}j(x,\infty)dx, \label{} \\
	   I_4(\gamma) &= j(\gamma,\infty)\left( \gamma m(\gamma)G_m(\gamma) -(1/2)G_m(\gamma)^2 - \int_{0}^{\gamma}xm(x)^2dx \right). \label{}
   \end{align}
   From Karamata's theorem and \cite[Proposition 1.5.9a]{Regularvariation}, we can see $I_3(\gamma) = o(I_1(\gamma)) \quad (\gamma \to \infty)$ and
   \begin{align}
	   I_4(\gamma) = \left\{ 
	   \begin{aligned}
		   &O(I_3(\gamma)) & (\alpha > 2), \\
		   &o(I_3(\gamma)) & (\alpha = 2),
	   \end{aligned}
	   \right. \quad (\gamma \to \infty). \label{}
   \end{align}
   Thus it follows that 
   \begin{align}
	   N(\gamma) \sim I_1(\gamma) + I_2(\gamma) \quad (\gamma \to \infty). \label{}
   \end{align}

   First, we consider the case of $\alpha > 2$.
   In this case, again from Karamata's theorem, it holds
   \begin{align}
	   m(x)^2\int_{0}^{x}j(y,\infty)dy \sim \frac{\alpha}{\alpha - 2}xm(x)^2j(x,\infty) \sim \frac{\alpha}{(\alpha -2)(\alpha - 1)^2} \frac{K(x)^2L(x)}{x} \quad (x \to \infty). \label{}
   \end{align}
   Similarly, it holds
   \begin{align}
	   m(x)G_m(x)j(x,\infty)
	   &\sim  \frac{\alpha}{(\alpha - 1)^2}\frac{K(x)^2L(x)}{x} \quad ( x \to \infty). \label{}
   \end{align}
   Since we are assuming $N(\infty) = \infty$, we obtain \eqref{eq219} for $\alpha > 2$.

   Next we consider the case $\alpha = 2$.
   Again from \cite[Proposition 1.5.9a]{Regularvariation}, we have
   \begin{align}
	   \lim_{x \to \infty}\frac{m(x)G_m(x)j(x,\infty)}{m(x)^2\int_{0}^{x}j(y,\infty)dy} = 0. \label{}
   \end{align} 
   Thus it follows $I_2(\gamma) = o(I_1(\gamma)) \quad (\gamma \to \infty)$,
   and we obtain \eqref{eq219} for $\alpha = 2$.
\end{proof}

\section{Examples} \label{section:examle}

Let us consider a unilateral jumping-in diffusion $X$ whose local generator $L$ on $(0,\infty)$ is
\begin{align}
	L = \frac{1}{2}\left( \frac{d^2}{dx^2} + b(x)\frac{d}{dx} \right). \label{eq199}
\end{align}
We assume $b(x)$ can be represented as 
\begin{align}
	b(x) = \frac{\delta - 1 + \epsilon(x)}{x} + \eta(x) \quad (x > 0) \label{eq200}
\end{align}
for some $\delta \in \bR$ and some measurable functions $\epsilon$ and $\eta$ satisfying
\begin{align}
	\lim_{x \to \infty}\epsilon(x) = 0, \quad \lim_{x \to \infty}\int_{1}^{x}\eta(y)dy = \bar{\eta} \in \bR. \label{}
\end{align}
When $b(x) = (\delta - 1)/x$, that is, when $\epsilon(x) = \eta(x) = 0 \ (x > 0)$, the local generator $L$ is equal to that of Bessel process of dimension $\delta$. Thus, we call jumping-in diffusions with a drift of the form \eqref{eq200} {\it jumping-in diffusions with a Bessel-like drift}. In Kasahara and Kotani \cite{KasaharaKotani:bessel-like}, they have studied a one-dimensional diffusion (without jumps) with the same form of drift and have shown the relation between the asymptotic behavior of the drift coefficient and that of the speed measure.

Define for an arbitrary fixed $x_0 > 0$,
\begin{align}
	W(x) := \exp \left(\int_{x_0}^{x}b(y)dy\right) = x^{\delta - 1}c(x)\exp \left( \int_{x_0}^{x}\frac{\epsilon(y)}{y}dy\right) \label{eq206}
\end{align}
with $c(x) = \exp\int_{x_0}^{x}\eta(y)dy$.
Let us assume $\delta < 0$.
For the generator $L$ in \eqref{eq199}, we define a corresponding string $\tilde{m}$ and its scale function $\tilde{s}$ by
\begin{align}
	\tilde{m}(x) = 2\int_{1}^{x}W(y)dy, \quad \tilde{s}(x) = \int_{0}^{x}\frac{dy}{W(y)}, \label{}
\end{align}
and then it follows that
\begin{align}
	A = \frac{d}{d\tilde{m}}\frac{d}{d\tilde{s}}. \label{}
\end{align}
Under the natural scale, the speed measure $m$ is given by $m(x) := \tilde{m}(\tilde{s}^{-1}(x))$.

To apply our main results, we need to check when the tail of $m$ varies regularly at $\infty$.
Since the function $W$ of the form \eqref{eq206} varies regularly at $\infty$ with exponent $\delta - 1$ (see e.g., \cite[Theorem 1.3.1]{Regularvariation}),
we set
\begin{align}
	\ell(x) := c(x)\exp \left( \int_{x_0}^{x}\frac{\epsilon(y)}{y}dy \right). \label{}
\end{align}
Then $W(x) = x^{\delta -1}\ell(x)$ and $\ell$ is a slowly varying function at $\infty$.
Note that since it holds that $\tilde{m}'(x) = 2W(x)$ and $\tilde{s}'(x) = 1/W(x)$, we have
\begin{align}
	m'(x) = \tilde{m}'(\tilde{s}^{-1}(x))(\tilde{s}^{-1}(x))' = 2W(\tilde{s}^{-1}(x))^2. \label{}
\end{align}
From Karamata's theorem, the function $\tilde{s}(x)$ varies regularly at $\infty$ with exponent $2 - \delta$ and therefore $\tilde{s}^{-1}(x)$ varies regularly at $\infty$ with exponent $1 / (2 - \delta)$ (see e.g., \cite[Proposition 1.5.15]{Regularvariation}).
Then from \cite[Proposition 1.5.14]{Regularvariation} and some calculation, we have
\begin{align}
	m(x,\infty) \sim (\alpha - 1)^{-1}x^{1/\alpha - 1}K(x) \quad (x\to \infty) \label{}
\end{align}
for $\alpha := 1 - 2/\delta$ and a slowly varying function $K$ at $\infty$ given by $K(x) := 2^{1/\alpha}  \alpha^{1/\alpha - 1} n^{\sharp} (x)^{1/\alpha}$.
Here $n^{\sharp}$ is a de Bruijn conjugate of $n(x) := \ell(x^{1/(2\alpha)})^{-1}$.
Hence $m$ satisfies $\mathrm{(M)}_{\alpha,K}$.

To apply Theorems \ref{informal-main-alpha2+} and \ref{informal-main-alpha-int}, we need to take appropriate slowly varying functions $u$ and $v$ satisfying \eqref{eq174}.
To obtain the asymptotic behavior of $K$, we consider a specific example.
Let $s \in \bR$. Taking $\epsilon(x) = (s-1) / \log x$ and $\eta(x) = 0$, it holds
\begin{align}
\ell(x) = \left(\frac{\log x}{\log x_0}\right)^{s-1} \quad (x > 0) \label{eq207}.
\end{align}
From \cite[p.433]{Regularvariation}, it holds that $\ell^{\sharp}(x) \sim \ell(x)^{-1} \ (x \to \infty)$ and, thus it follows $n^{\sharp}(x) \sim (2\alpha)^{-s+1}\ell(x) \ (x \to \infty)$.
Thus, by taking $x_0$ appropriately, it holds 
\begin{align}
	K(x) = (\log x)^{(s-1) / \alpha}. \label{eq222}
\end{align}
Let
\begin{align}
	j(dx) := \frac{\alpha}{2} \left( \frac{1}{x^{a + 1}} \wedge \frac{(\log x)^{t-1}}{x^{2/\alpha + 1}}\right)dx \label{}
\end{align}
for $a \in (0,1/\alpha)$ and $t \in \bR$.
It is not difficult to see $(m,j)$ satisfies $\mathrm{(C)}$ and $\mathrm{(J)}_{2/\alpha,L}$ for 
\begin{align}
	L(x) = (\log x)^{t-1}. \label{eq223}
\end{align}
Since we would like to apply Proposition \ref{comparison of N,K,L alpha > 2}, we take $s$ and $t$ so that $N(\infty) = \infty$. For example, here we take
\begin{align}
	\left\{
	\begin{aligned}
		&2(s-1) / \alpha + t \geq 0 & (\alpha > 2), \\
		&t > 0 \quad \text{and} \quad s+t > 0 & (\alpha = 2).
	\end{aligned}
	\right. \label{}
\end{align}
Then from Proposition \ref{comparison of N,K,L alpha > 2}, it holds
for $\alpha > 2$ that
\begin{align}
	N(\gamma) \sim  C_{\alpha}(\log \gamma)^{\frac{2(s-1)}{\alpha} + t}
	\quad (\gamma \to \infty), \label{}
\end{align}
where $C_{\alpha} := \frac{\alpha}{(\alpha -1)(\alpha -2)}\left( \frac{2(s - 1)}{\alpha} + t \right)$,
and for $\alpha = 2$ it holds that
\begin{align}
	N(\gamma) \sim C_2 
	(\log \gamma)^{s+t} \quad (\gamma \to \infty), \label{}
\end{align}
where $C_2 := \frac{1}{t(s+t)}$.

Thus in the case of $\alpha > 2$, when we define
\begin{align}
	u(\gamma) := \sqrt{c_1}(\log \gamma)^{((s-1)/\alpha) + \eps/2} \quad
	\text{and} \quad
	v(\gamma) := c_2(\log \gamma)^{t-1 + \eps} \label{}
\end{align}
for $\eps \in (0,1)$ and constants $c_1,c_2 > 0$ such that $c_1c_2 = C_{\alpha}$, we may see all the assumptions of Theorem \ref{convtoBM2} holds and therefore the convergence \eqref{eq218} holds.

For $\alpha = 2$, we may take
\begin{align}
	u(\gamma) := \sqrt{c_1} (\log \gamma)^{s/2} \quad \text{and} \quad 
	v(\gamma) := c_2(\log \gamma)^t \quad (\gamma \geq 1) \label{}
\end{align}
for constants $c_1,c_2 > 0$ such that $c_1c_2 = C_2$. 

\appendix

\section{Appendix: Proof of Theorem \ref{informal-main-alpha-int}} \label{appendix}

We give a proof of Theorem \ref{informal-main-alpha-int}.
Since it is similar to that of Theorem \ref{informal-main-alpha2+}, we focus on the different parts.


\begin{proof}[Proof of Theorem \ref{informal-main-alpha-int}]
	We may assume $m(\infty) = 0$ without loss of generality.
	Define $m_\gamma$ and $j_\gamma$ as \eqref{eq225}.
	We show $\{ m_\gamma \}_\gamma$ and $\{ j_\gamma \}_\gamma$ satisfy the conditions of Theorem \ref{convtoBM2} for $\kappa = 1$.

    The equality $\lim_{\gamma \to \infty}\int_{1}^{\infty}G_{m_\gamma} (x)j_\gamma(dx) = 0$ and the conditions (ii) and (iv) of Theorem \ref{convtoBM2} follow from the same arguments in the proof of Theorem \ref{informal-main-alpha2+}, combining the assumption (v) when $\alpha = 2$.
	In addition, the condition (iii) of Theorem \ref{convtoBM2} is obvious.
	Thus, to apply Theorem \ref{convtoBM2}, it is enough to check $m_\gamma \overset{G}{\to} 0 \quad (\gamma \to \infty)$. 
    We only show
	\begin{align}
		\lim_{\gamma \to \infty}\int_{0}^{1}G^{d}_{m_\gamma}(x)dm_\gamma(x) = 0 \label{eq216d}
	\end{align}
	for some $d$.
	Set $N := \alpha - 1$ and fix an integer $d \geq d(m) \vee N$ so that
    \begin{align}
		\lim_{\gamma \to \infty}\frac{K(\gamma)^{d - N}
		\int_{1}^{\gamma}\frac{K(x)^{\alpha}}{x}dx}{u(\gamma)^{d+1}}
		= 0. \label{}
    \end{align} 
	Define the function $H^k \ (k \geq 1)$ as in \eqref{H-func}.
	Then for $k \geq N$, by the same argument, we have \eqref{eq210} and \eqref{eq213} for a constant $C_k > 0$ and the function $M$ in \eqref{eq226}.
	From \eqref{eq215} and \cite[Proposition 1.5.9]{Regularvariation}, 
	it holds for an integer $\alpha$
	\begin{align}
		M(\gamma) = O\left(\int_{1}^{\gamma}\frac{K(x)^{\alpha}}{x}dx\right) \quad ( \gamma \to \infty). \label{eq221d}
	\end{align}
	Then we can show \eqref{eq227} in the same way, that is,
	\begin{align}
		\int_{0}^{1}G^d_{m_\gamma}(y)dm_{\gamma}(y) 
		\leq O\left(\frac{K(\gamma^{\alpha/2})^{d-N}M(\gamma^{\alpha/2})}{u(\gamma^{\alpha/2})^{d+1}}\right). \label{}
	\end{align}
	From the assumption (iv), we obtain \eqref{eq216d}, and it follows $m_\gamma \overset{G}{\to} 0$.
	The asymptotic behavior \eqref{eq217d} follows from the exactly same way in the proof of Theorem \ref{informal-main-alpha2+}.
\end{proof}


\bibliography{Bibliography.bib} 

\begin{thebibliography}{1}

\bibitem{Regularvariation}
N.~H. Bingham, C.~M. Goldie, and J.~L. Teugels.
\newblock {\em Regular variation}, volume~27 of {\em Encyclopedia of
  Mathematics and its Applications}.
\newblock Cambridge University Press, Cambridge, 1987.

\bibitem{Feller:Theparabolic}
W.~Feller.
\newblock The parabolic differential equations and the associated semi-groups
  of transformations.
\newblock {\em Ann. of Math. (2)}, 55:468--519, 1952.

\bibitem{Ito:Essentialsof}
K.~It\^o.
\newblock {\em Essentials of stochastic processes}, volume 231 of {\em
  Translations of Mathematical Monographs}.
\newblock American Mathematical Society, Providence, RI, 2006.
\newblock Translated from the 1957 Japanese original by Yuji Ito.

\bibitem{Ito:PPP}
K.~It\^o.
\newblock {\em Poisson point processes and their application to {M}arkov
  processes}.
\newblock SpringerBriefs in Probability and Mathematical Statistics. Springer,
  Singapore, 2015.
\newblock Mimeographic original in 1969.

\bibitem{KasaharaKotani:bessel-like}
Y.~Kasahara and S.~Kotani.
\newblock Diffusions with {B}essel-like drifts.
\newblock {\em Kyoto J. Math.}, 55(4):773--797, 2015.

\bibitem{KasaharaWatanabe:Brownianrepresentation}
Y.~Kasahara and S.~Watanabe.
\newblock Brownian representation of a class of {L}\'evy processes and its
  application to occupation times of diffusion processes.
\newblock {\em Illinois J. Math.}, 50(1-4):515--539, 2006.

\bibitem{Rogers1}
L.~C.~G. Rogers.
\newblock It\^{o} excursion theory via resolvents.
\newblock {\em Z. Wahrsch. Verw. Gebiete}, 63(2):237--255, 1983.

\bibitem{YamatoYano}
K.~Yamato and K.~Yano.
\newblock Fluctuation scaling limits for positive recurrent jumping-in
  diffusions with small jumps.
\newblock {\em J. Funct. Anal.}, 279(7):108655, 33, 2020.

\bibitem{Yano:Convergenceofexcursion}
K.~Yano.
\newblock Convergence of excursion point processes and its applications to
  functional limit theorems of {M}arkov processes on a half-line.
\newblock {\em Bernoulli}, 14(4):963--987, 2008.

\end{thebibliography}
\bibliographystyle{plain}

\end{document}